\documentclass[twoside]{article}
\usepackage[left=2cm,right=2cm,top=2cm,bottom=2cm]{geometry}
\usepackage{multicol}
\usepackage{amsmath}
\usepackage{amsfonts}
\usepackage{amssymb}
\usepackage{graphicx}
\usepackage{color}
\usepackage[utf8]{inputenc}
\usepackage{amsfonts}
\usepackage{color}
\usepackage[normalem]{ulem}
\usepackage{nomencl}
\makenomenclature

\newenvironment{proof}{\noindent{\sc Proof.}}{\qed}
\newtheorem{theorem}{Theorem}[section]
\newtheorem{lemma}{Lemma}[section]
\newtheorem{cor}{Corollary}[section]
\newtheorem{rem}{Remark}[section]
\newtheorem{definition}{Definition}[section]
\newtheorem{prop}{Proposition}[section]

\newcommand{\qed}{$\blacksquare$}
\def\hindu{\arabic}

\renewcommand{\theequation}{\hindu{section}.\hindu{equation}}
\def\bhag#1{\noindent
\setcounter{equation}{0}
\section{#1}
}

\def\HH{{\mathbb H}}
\def\NN{{\mathbb N}}
\def\RR{{\mathbb R}}
\def\CC{{\mathbb C}}
\def\ZZ{{\mathbb Z}}
\def\SS{{\mathbb S}}

\def\TT{\mathbb T}

\def\a{\alpha}

\def\x{\mathbf{x}}
\def\k{\mathbf{k}}
\def\y{\mathbf{y}}

\def\w{\mathbf{w}}
\def\z{\mathbf{z}}

\def\O{{\cal O}}

\def\C{{\mathcal C}}

\def\be{\begin{equation}}
\def\ee{\end{equation}}
\def\bea{\begin{eqnarray}}
\def\eea{\end{eqnarray}}
\def\eref#1{(\ref{#1})}
\def\disp{\displaystyle}

\def\donchitre#1#2{\vskip 6.5cm\noindent
\parbox[t]{1in}{\special{eps:#1.eps x=6.5cm y=5.5cm}}
\hbox to 7cm{}\parbox[t]{0.0cm}{\special{eps:#2.eps x=6.5cm y=5.5cm}}}

\def\tn{|\!|\!|}
\def\XX{{\mathbb X}}
\def\BB{{\mathbb B}}

\def\gs{\gtrsim}
\def\ls{\lesssim}

\title{Local approximation of operators}
\author{
 H.~N.~Mhaskar\thanks{
Institute of Mathematical Sciences, Claremont Graduate University, Claremont, CA 91711. 
\textsf{email:} hrushikesh.mhaskar@cgu.edu.
The research is supported in part by NSF grant DMS 2012355 and ARO Grant W911NF2110218.}
 }
 \date{}
\begin{document}
\maketitle

\begin{abstract}
Many applications, such as system identification, classification of time series, direct and inverse problems in partial differential equations, and uncertainty quantification lead to the question of approximation of a non-linear operator between metric spaces $\mathfrak{X}$ and $\mathfrak{Y}$. We study the problem of determining the degree of approximation of  such operators on a compact subset $K_\mathfrak{X}\subset \mathfrak{X}$  using a finite amount of information.
If $\mathcal{F}: K_\mathfrak{X}\to K_\mathfrak{Y}$, a well established strategy to approximate $\mathcal{F}(F)$ for some $F\in K_\mathfrak{X}$ is to encode $F$ (respectively, $\mathcal{F}(F)$) in terms of a finite number $d$ (repectively $m$) of real numbers. Together with appropriate reconstruction algorithms (decoders), the problem reduces to the approximation of $m$ functions on a compact subset of a high dimensional Euclidean space $\RR^d$, equivalently, the unit sphere $\SS^d$ embedded in $\RR^{d+1}$. 
The problem is challenging because $d$, $m$, as well as the complexity of the approximation on $\SS^d$ are all large, and it is necessary to estimate the accuracy keeping track of the inter-dependence of all the approximations involved. 
In this paper, we establish constructive methods to do this efficiently; i.e., with the constants involved in the estimates on the approximation on $\SS^d$ being $\O(d^{1/6})$. 
We study different smoothness classes for the operators, and also propose a method for approximation of $\mathcal{F}(F)$ using only information in a small neighborhood of $F$, resulting in an effective reduction in the number of parameters involved. 
To further mitigate the problem of large number of parameters, we propose prefabricated networks, resulting in a substantially smaller number of effective parameters.
The problem is studied in both deterministic and probabilistic settings.
\end{abstract}


\noindent\textbf{AMS Classification 2000:} 41A25, 41A63, 42C10.\\

\bhag{Introduction}\label{bhag:introduction}

While much of approximation theory deals with the question of approximation of functions on subsets of a Euclidean space, many applications require an approximation of non-linear functionals and even non-linear operators defined on compact subsets of function spaces. 
We give a few examples here; many more are listed in the references cited in Section~\ref{bhag:relatedwork}.
\begin{enumerate}
\item 
In system identification problems, the hidden states of a non-linear system are not known, but need to be modeled using observations of the input-output relationships of the system (e.g., \cite{narendra1990identification, levin1995identification}). 
Both the input signals and the output signals are functions of time, and the model is thus an unknown non-linear operator which needs to be approximated.
\item
In prediction of time series, a time series $(t_\ell)_{\ell=0}^\infty$ is modeled by a functional relationship
$$
t_{\ell+q}=F(t_\ell, \cdots, t_{\ell+q-1}), \qquad \ell=0,1,\cdots,
$$
 for a judiciously chosen $q$, where $F$ is some possibly non-linear function (e.g., \cite{narendra1990identification, gaussian_diabetes}). 
 In order to use such a model for classification of time series, different time series are modeled by different functions $F$, so that the class label is a non-linear functional on a class of functions.
\item
In the theory of non-linear partial differential equations (PDEs), the mapping from the initial/boundary conditions to the  solution or in the case of inverse problems, from the solution to the initial/boundary conditions or the coefficient functions in the differential equations are all non-linear operators.
\item
In uncertainty quantification problems, the coefficients of the PDEs are random functions of the variables, and one is interested in some quantity of interest (e.g., \cite{cohen2011analytic}). 
Clearly, the quantity of interest is a possibly non-linear functional on the space of functions involved.
\end{enumerate}

An obvious and natural way to solve such problems is to encode the input functions as well as output functions using finitely many parameters, and treat the problem as a problem of approximation of functions between finite dimensional Euclidean spaces. 
For example, a time series can be encoded in a variety of ways, such as thresholded PCA components of  snippets \cite{mason2021manifold}, values of the empirical mode decomposition at certain points \cite{mert2018emotion}, parameters of an ARMA model \cite{chellappa}, etc.
In the context of PDEs, the input and output functions can be encoded using values of the functions involved \cite{lu2019deeponet},  coefficients with respect to certain frames/bases \cite{dung2022analyticity}, random features \cite{kovachki2021neural}, etc.
Clearly, any encoder needs to be associated with a decoder which can approximate the input and output functions well\footnote{An encoder is referred to also as a parameter selection, feature map, or information operator. 
A decoder is referred to also as approximation operator, or reconstruction algorithm.}.
The exact nature of the encoder and decoder need to depend upon the specific application, but
it is  clear intuitively that the number of parameters defined by the encoders must be very large in order for the corresponding decoders to achieve a good approximation. 
Therefore, even though the idea behind the reduction to the problem of operator approximation to that approximation of functions between Euclidean spaces is obvious, the dimensions of these spaces poses a formidable problem.
In this paper, we  assume that the appropriate encoder/decoder pairs have been selected, and focus on the technical problems arising from the high dimensionality of the Euclidean spaces to which the parameters generated/desired by the encoders  belong.
We will discuss the issues involved in Section~\ref{bhag:techintro}.

After pointing out a few papers related to the current paper in Section~\ref{bhag:relatedwork}, we will formalize in Section~\ref{bhag:operatorapprox}  the intuitive thinking just described in an abstract manner.
We will explain there how the problem can be reduced to the problem of efficient approximation of functions on a high dimensional sphere, and highlight the technical contributions of the current paper.
In the rest of the paper, we will focus on approximation on a high dimensional sphere. 
In Section~\ref{bhag:notation}, we will review some preparatory material required both to formulate our results and to prove them.
In Section~\ref{bhag:mainresults}, we will formulate our main theorems about approximation on the sphere.
In Section~\ref{bhag:manifold}, we illustrate an example on the application of our theory for approximation of operators defined on the space of continuous functions on a smooth, compact, Riemannian manifold.
The proofs of all the new results in this paper are given in Section~\ref{bhag:proofs}. In Appendix~\ref{bhag:computation}, we make some comments on a possible computational scheme for the kernels introduced in Section~\ref{bhag:kernels}. 
In Appendix~\ref{bhag:non_chen}, we give another example to illustrate a certain technical point which could not be commented upon in the main part of the paper.

\bhag{Related work}\label{bhag:relatedwork}
An early and widely cited work on the problem of functional/operator approximation is the paper \cite{chen1995universal} by   Chen and Chen. They consider approximation by neural networks for nonlinear operators using what is sometimes called trunk and branch networks. First, one obtains an approximation of $\mathcal{F}(F)(y)\approx\sum_k \beta_k(F)\tau_k(y)$, where $\tau_k$ is a basis for range space of $\mathcal{F}$. The \textit{trunk networks} are neural network approximations to $\tau_k$, and the \textit{branch networks}  approximate the coefficients $\beta_k(F)$. In turn, each branch network is a composition of two networks, one to approximate $F$ and one to approximate $\beta_k$ as a function of the parameters of the network to approximate $F$. 
The authors prove universal approximation theorems for the resulting networks.
While this approach is adequate for universal approximation theorems, it introduces extra error terms due to the approximation of $\mathcal{F}(F)$ in the indicated format (see Appendix~\ref{bhag:non_chen}).

From the point of view of system identification, the problem was studied already in early works of Sandberg \cite{sandberg1991approximation, sandberg1991gaussian}, Modha and Hecht-Nielsen \cite{modha1993multilayer}, Dingankar \cite{dingankar1995applications}, among others.
This work motivated our own work \cite{neursyst}.

In \cite{neursyst}, we have taken a simpler approach than that of \cite{chen1995universal} for the approximation of nonlinear functionals (such as $\beta_k$) on spaces of the form $L^p([-1,1]^s)$. We have constructed networks with a single hidden layer by considering the functional as a function of the coefficients of $F$ in a tensor product Legendre polynomial series. We have proved estimates  on the degree of approximation in terms of the size of the networks, and proved that they are optimal in the sense of non-linear widths. The current paper is a substantial generalization and refinement of this work.
The paper \cite{song2021approximation}  obtains results similar to those in \cite{neursyst} in the case when the activation function is an ReLU function.

The problem of approximation of functions of infinitely many variables, especially on tensor product domains has a long history of research in the information based complexity community - there are too many papers in the Journal of Complexity alone to give a reasonably good bibliography.
A detailed treatment from this point of view can be found in the series of books by Novak and Wo\'zniakowski \cite{novak2008tractability}.
We point out only two recent papers. In
\cite{werschulz2021tractability}, Werschulz and Wo\'zniakowski study the tractability of approximating solutions of Volterra equations in high dimensions.
In
\cite{kritzer2020exponential}, Kritzer, Pillichshammer, and Wo\'zniakowski study the tractability of approximation of operators between tensor product weighted Hilbert spaces.

In the last couple of years, a great deal of interest in this direction is triggered by possible applications to the solutions of direct and inverse problems involving partial differential equations (PDEs). For example, in \cite{li2020fourier, nelsen2021random} the authors introduce the concept of Fourier neural networks, and examime the feature selection and training algorithms for solutions of PDEs.
The paper \cite{kovachki2021neural} establishes a universal approximation theorem for deep networks in the topology of uniform convergence on compact sets using Hilbert space norms. 
This paper has  a long list of related papers and the correspondingly long discussion.
Universal approximation property for deep networks is established also in \cite{kovachki2021universal, deng2021convergence}, where the rates of convergence are studied for special PDEs.
The paper
\cite{lanthaler2021error} follows the approach of Chen and Chen, approximating the functionals $\beta_k(F)$  using values of $F$. Error estimates are given in terms of an appropriate $L^2$ norm. Lower bounds are established on the degree of approximation and it is pointed out that the
 curse of dimensionality is avoided for holomorphic functions and solutions of certain PDEs. 
 We note that every compact set of functions on $\RR^q$ has a non-linear width, depending necessarily on $q$, usually increasing with $q$. 
 In 
\cite{liu2022deep}, the authors give statistical estimates on the error in approximation in the presence of noisy data for the solution of judiciously formulated optimization problems involved in training the networks used for approximation.

\bhag{Approximation of operators}\label{bhag:operatorapprox}
In this section, we wish to formulate the problem of approximation of operators in an abstract setting. 
Before doing so in Section~\ref{bhag:formulation}, we illustrate the main technical problems by means of a simple example in Section~\ref{bhag:techintro}.
After discussing in Section~\ref{bhag:transformation}  the transformation of the problem of approximating functions on $\RR^d$ to that of approximating functions on the unit sphere $\SS^d$ embedded in $\RR^{d+1}$,
we discuss the main technical contributions of the current paper in Section~\ref{bhag:contributions}.

\subsection{An elementary example}\label{bhag:techintro}

The purpose of this section is illustrate the issues that arise in our approach to the approximation of operators.
Although we will strive to keep the notation consistent with the rest of the  paper, the notation here will expire at the end of this subsection.

Let $\TT=\RR/(2\pi\ZZ)$, so that functions on $\TT$ are exactly the $2\pi$-periodic function on $\RR$. 
The  space $L^2$ comprises measurable functions $F : \TT\to \CC$ for which
$$
\|F\|=\left\{\frac{1}{2\pi}\int_{-\pi}^{\pi} |F(t)|^2dt\right\}^{1/2} <\infty,
$$
where two functions are considered equal if they are equal almost everywhere with respect to the Lebesgue measure.
In this example, we are interested in approximating a possibly non-linear  operator $\mathcal{F} :L^2\to L^2$ that satisfies
\be\label{eq:introlip}
\|\mathcal{F}(F_1)-\mathcal{F}(F_2)\|\le \|F_1-F_2\|, \qquad F_1, F_2\in L^2.
\ee
For $F\in L^2$, we define
$$
\hat{F}(k)=\frac{1}{2\pi}\int_{-\pi}^\pi F(t)e^{-ikt}dt, \qquad k\in\ZZ.
$$
For integer $n\ge 1$, the Fourier partial sum operator is defined by
$$
s_n(F)(x)=\sum_{|k|<n}\hat{F}(k)\exp(ikx).
$$
It is well known that $\lim_{n\to\infty}s_n(F)= F$ in $L^2$ so that for every $n$, an information operator (encoder)  is defined by $\mathcal{I}_n(F)=(\hat{F}(k))_{|\k|<n}\in \RR^n$, and the corresponding reconstruction algorithm (decoder) is defined by $\mathcal{A}_n((a_k)_{|k|<n})=\sum_{|k|<n}a_k\exp(ik\circ)$.
In particular, $s_n(F)=\mathcal{A}_n(\mathcal{I}_n(F))$.

Although the operator $\mathcal{F}$ itself is defined on $L^2$, we are interested in approximating it only on the compact subset $K\subset L^2$ comprising $F\in L^2$ for which
\be\label{eq:l2lipnorm}
\|F\|_L=\|F\|+\max_{0<h<1}\frac{\|F(\circ+h)-F(\circ)\|}{h} \le 1.
\ee
It is well known that there exists an absolute constant $c_1$ such that
\be\label{eq:lipdegapprox}
\sup_{F\in K}\|F-s_n(F)\|\le c_1/n, \qquad n\in\NN.
\ee
Since $\mathcal{F}$ is continuous, it maps $K$ into another compact set, say $K_1$. 
Then it is well known that there exists a non-increasing sequence $\{\delta_m\}_{m=1}^\infty$ converging to $0$ as $m\to\infty$ such that
\be\label{eq:compactdegapprox}
\sup_{G\in K_1}\|G-s_m(G)\|\le \delta_m, \qquad m\in\NN.
\ee
Thus, for any $F\in K$, we obtain using \eqref{eq:compactdegapprox}, \eqref{eq:introlip}, and \eqref{eq:lipdegapprox} that
\be\label{eq:triang}
\begin{aligned}
\|\mathcal{F}(F)-s_m(\mathcal{F}(s_n(F)))\|&\le \|\mathcal{F}(F)-s_m(\mathcal{F}(F))\|+\| s_m(\mathcal{F}(F))-s_m(\mathcal{F}(s_n(F)))\|\\
&\le \delta_m+\|\mathcal{F}(F)-\mathcal{F}(s_n(F))\|\le \delta_m+\|F-s_n(F)\|\\
&\le \delta_m+c_1/n.
\end{aligned}
\ee
The estimate \eqref{eq:triang} is entirely due to our set up and well known results on approximation by trigonometric polynomials to the function classes under consideration.
The main problem here to learn $\mathcal{F}$ from a training data comprising a number of functions $F$, more precisely, the features of the functions $F$, given by their Fourier coefficients.
We note that 
$$
s_m(\mathcal{F}(s_n(F)))=\mathcal{A}_m\left(\mathcal{I}_m(\mathcal{F}(\mathcal{A}_n(\mathcal{I}_n(F)))\right).
$$
Out of these, $\mathcal{A}_m$ is ``domain knowledge'', $\mathcal{I}_n(F)\in\RR^{2n-1}$ is the ``input variable''. 
Different points in $\RR^{2n-1}$ will correspond to different functions (actually, trigonometric polynomials).
Thus the problem reduces to approximation of the map from $\RR^{2n-1}\to \RR^{2m-1}$ given by
\be\label{eq:intromainfn}
\mathbf{f}(\mathbf{x})= \mathcal{I}_m(\mathcal{F}(\mathcal{A}_n(\x))).
\ee
We observe that even though $\mathcal{F}(F)$ may not be in $K$, \eqref{eq:introlip} can be used to show that $\mathbf{f}$ is a Lipschitz continuous function. 
Given that for every $F\in K$, $\mathcal{I}_n(F)\in \BB_d$, where $d=2n-1$, and $\BB_d$ is the unit ball of $\RR^d$, we may approximate each component of $\mathbf{f}$ and focus on approximation of complex valued Lipschitz continuous functions on $\BB_d$.
There are many results known for such an approximation using various kind of approximations.
Typically, with an approximation $\mathbb{G}(f)$ of $f$ involving $N$ parameters, one has an estimate of the form
$$
\|f-\mathbb{G}(f)\|\le \frac{c_2(d)}{N^{1/d}}.
$$
Thus, using the $N$ parameter approximation process $\mathbb{G}$ for each component of $\mathbf{f}$, and abusing the notation somewhat, we obtain in the end
\be\label{eq:introfinal}
\|\mathcal{F}(F)-\mathbb{G}(s_n(F))\|\le \delta_m +c/n+\frac{\sqrt{m}c_2(d)}{N^{1/d}}.
\ee
It is clear that to get a good approximation, $m$ and $n$ (and hence, $d=2n-1$), as well as $N$ should be large. 
The first two terms on the right hand side of \eqref{eq:introfinal} are characteristic of our problem. 
The main techincal difficulty is to control the last term.
Here, the factor $\sqrt{m}$ comes because the process $\mathbb{G}$ is applied to each component of $\mathbf{f}$, and an estimate in the sense of $L^2$ norm is desired.
We don't consider this to be a major technical problem - it is only a question of the norm involved.
Under the assumption of Lipschitz continuity alone,  the term $N^{-1/d}$ is unimprovable as well.
So, the main technical challenge is how to control $c_2(d)$ as a function of $d$.
The other technical challenge is how to work efficiently with the large number of input variables and the parameters in the process $\mathbb{G}$; in particular, how to reduce the space and time complexity of the process.

In this paper, we will address these issues as follows.
First, unlike the currently used approaches, we will separate the pre-processing and post-processing steps involved in the construction of information operators (encoders) and reconstruction algorithms (decoders) from the main technically challenging question of approximating the components of $\mathbf{f}$ as in \eqref{eq:intromainfn}.
Second, we will transform this problem into a problem of approximation on the unit  sphere in $\RR^{d+1}$.
Then we will define higher smoothness conditions on $\mathcal{F}$ (equivalently, components of $\mathbf{f}$) to improve upon the factor $N^{-1/d}$.
Again, the central technical problem is to define the smoothness as well as the operators so as keep the constant $c_2(d)$ dependent sub-polynomially on $d$.
Finally, we design our approximation process so that in approximating $\mathbf{f}(\x)$ for some point $\x$ (equivalently, an input function $F$), we need to utilize only those points which are in a small neighborhood of $\x$ (equivalently, input functions in a small neighborhood of $F$). 

\subsection{Problem formulation}\label{bhag:formulation}
Let $(\mathfrak{X},\rho_\mathfrak{X})$, $(\mathfrak{Y},\rho_\mathfrak{Y})$ be metric spaces, $K_{\mathfrak{X}}$, (respectively, $K_{\mathfrak{Y}}$) be a compact subset of $\mathfrak{X}$ (repsectively, $\mathfrak{Y}$), $\mathcal{F}:K_\mathfrak{X}\to K_{\mathfrak{Y}}$ be a continuous function.
The goal is to approximate $\mathcal{F}$, with estimates on the accuracy in approximation.
The idea is the following.

In the theory of optimal recovery \cite{micchelli1985lectures, micchelli1977survey, devore1989optimal}, it is customary to define an approximation as a composition $\mathcal{A}_{d,K_{\mathfrak{X}}}\circ \mathcal{I}_{d,K_{\mathfrak{X}}}$ of two (possibly nonlinear) operators: The \textit{information operator} (\textit{encoder}) $\mathcal{I}_{d,K_{\mathfrak{X}}}:K_{\mathfrak{X}}\to \RR^d$, and the \textit{(reconstruction) algorithm} (\textit{decoder}) $\mathcal{A}_{d,K_{\mathfrak{X}}} :\RR^d\to \mathfrak{X}$. 

The worst case error with these operators is defined to be
\be\label{eq:worst_error}
\mathsf{wor}(K_{\mathfrak{X}};\mathcal{A}_{d,K_{\mathfrak{X}}}, \mathcal{I}_{d,K_{\mathfrak{X}}})=\sup_{F\in K_{\mathfrak{X}}}\rho_\mathfrak{X}\bigg(F, \mathcal{A}_{d,K_{\mathfrak{X}}}\big(\mathcal{I}_{d,K_{\mathfrak{X}}}(F)\big)\bigg).
\ee
The best choice for $\mathcal{I}_{d,K_{\mathfrak{X}}}$, $\mathcal{A}_{d,K_{\mathfrak{X}}}$ is described by
the non-linear $d$-width of $K_{\mathfrak{X}}$ defined by
\be\label{eq:mwidth}
\mathsf{width}_{d,\mathfrak{X}}(K_{\mathfrak{X}})=\inf_{\mathcal{I}_{d,K_{\mathfrak{X}}}, \mathcal{A}_{d,K_{\mathfrak{X}}}}\mathsf{wor}(K_{\mathfrak{X}};\mathcal{A}_{d,K_{\mathfrak{X}}}, \mathcal{I}_{d,K_{\mathfrak{X}}}),
\ee
where the infimum is over \textbf{all continuous} information operators and \textbf{all} algorithms.

In our theory, we assume that an appropriate (if not optimal) choice of the operators $\mathcal{I}_{d,K_{\mathfrak{X}}}$, $\mathcal{A}_{d,K_{\mathfrak{X}}}$, $\mathcal{I}_{m,K_{\mathfrak{Y}}}$, $\mathcal{A}_{m,K_{\mathfrak{Y}}}$ is already made based on domain knoweldge.
The problem of approximating $\mathcal{F}(F)$ for any $F\in K_{\mathfrak{X}}$ is thus reduced to the problem of approximating  $m$ functions of the form $f_j :\RR^d \to \RR$ so that $f_j(\mathcal{I}_{d,K_{\mathfrak{X}}}(F))$ is  the $j$-th component of 
$\disp\mathcal{I}_{m,K_{\mathfrak{Y}}}\left(\mathcal{F}(\mathcal{A}_{d,K_{\mathfrak{X}}}\left(\mathcal{I}_{d,K_{\mathfrak{X}}}(F)\right))\right)$.
This idea is presented pictorially in Figure~\ref{fig:oper_approx}.
\begin{figure}[h]
\begin{center}
\includegraphics[width=0.7\textwidth,keepaspectratio]{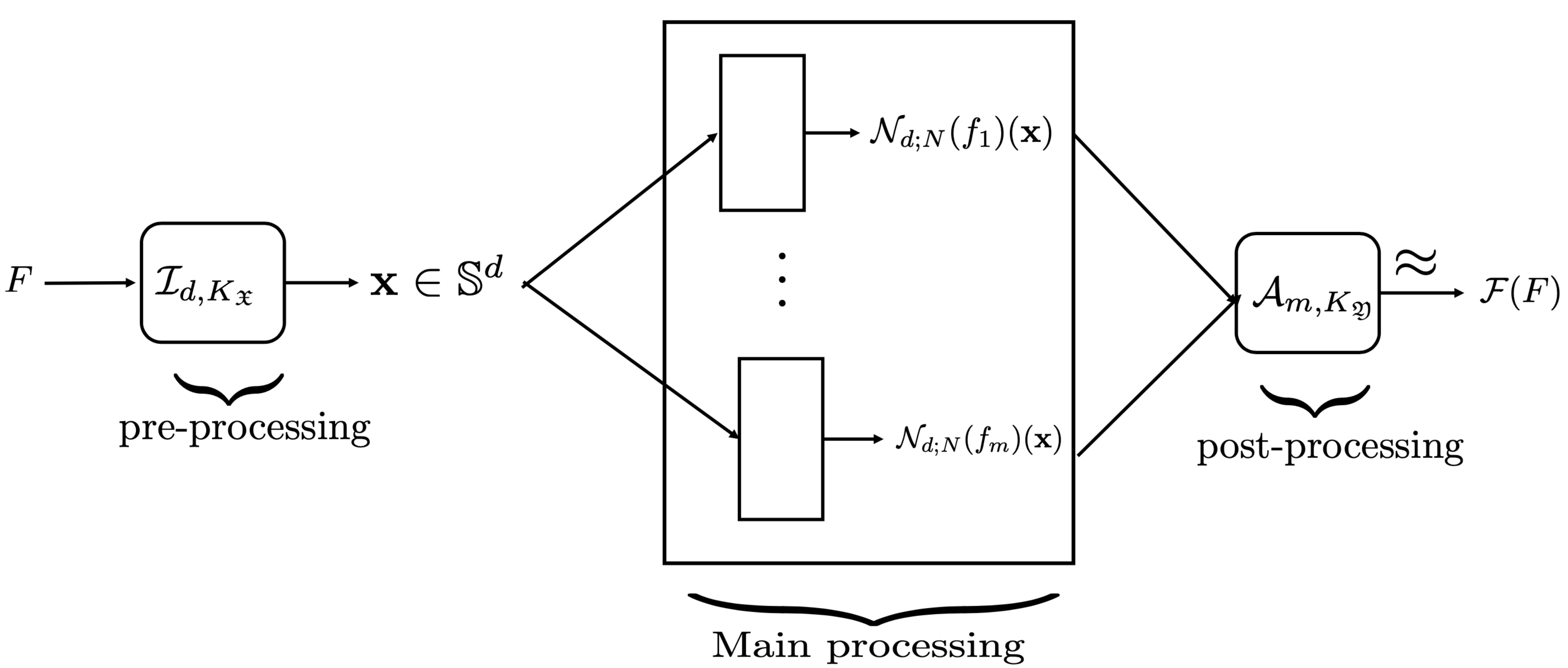} 
\end{center}
\caption{Schematics for operator approximation, see text for the explaination of symbols. The output of the box labeled $\mathcal{A}_{m,K_{\mathfrak{Y}}}$ is actually an approximation to $\mathcal{F}(F)$ as indicated in \eqref{eq:protoestimate}.}
\label{fig:oper_approx}
\end{figure}

To illustrate the error estimate in the approximation of $\mathcal{F}$ while keeping the notation relatively simple, let us assume that $\mathcal{F}$, $\mathcal{I}_{d,K_{\mathfrak{X}}}$, $\mathcal{A}_{d,K_{\mathfrak{X}}}$, $\mathcal{I}_{m,K_{\mathfrak{Y}}}$, $\mathcal{A}_{m,K_{\mathfrak{Y}}}$ are all Lipschitz continuous. Because of continuity and compactness,   there exists a compact subset $K_S\subset C_0(\RR^d)$ such that each $f_j\in K_S$. 
We emphasize that $K_S$ need not consist of the possible $f_j$'s alone; the only requirement is that they all be in $K_S$.
Suppose we find an approximation operator $\mathbb{G}_{d;N}: K_S\to C_0(\RR^d)$ depending upon $N$ parameters such that
\be\label{eq:sphwidth}
\sup_{f\in K_S}\|f-\mathbb{G}_{d;N}(f)\|_\infty \le \epsilon_{d,N},
\ee
then it is easy to deduce that
\be\label{eq:protoestimate}
\begin{aligned}
\sup_{F\in K_{\mathfrak{X}}}\rho_\mathfrak{Y}\bigg(\mathcal{F}(F), \mathcal{A}_{m,K_{\mathfrak{Y}}}&\big(\{\left(\mathbb{G}_{d;N}(f_j)(\mathcal{I}_{d,K_{\mathfrak{X}}}(F))\right)\}\big)\bigg)\\
&\le L\left\{\mathsf{wor}(K_{\mathfrak{X}};\mathcal{A}_{d,K_{\mathfrak{X}}}, \mathcal{I}_{d,K_{\mathfrak{X}}})+\epsilon_{d,N}+\mathsf{wor}(K_{\mathfrak{Y}};\mathcal{A}_{m,K_{\mathfrak{Y}}}, \mathcal{I}_{m,K_{\mathfrak{Y}}})\right\},
\end{aligned}
\ee
where $L$ depends upon the various Lipschitz constants involved.
A crude estimate of the number of parameters in the process is $dNm$.

\subsection{A transformation}\label{bhag:transformation}

In \eqref{eq:protoestimate}, the two worst case errors come from the application domain and the choice of $\mathcal{I}_{d,K_{\mathfrak{X}}}$, $\mathcal{A}_{d,K_{\mathfrak{X}}}$, $\mathcal{I}_{m,K_{\mathfrak{Y}}}$, $\mathcal{A}_{m,K_{\mathfrak{Y}}}$
In order to make the two $\mathsf{wor}$ expressions small, one needs to make $d$, $m$ large.
To keep $\epsilon_{d,N}$ under control, one has to make $N$ large also.
Thus, the problem involves an approximation of a large number $m$ of functions on a space of a large dimension $d$ by a class with high complexity $N$.
It is clear that the bottleneck is the estimation of $\epsilon_{d,N}$, so that it is critical to investigate the dependence of $\epsilon_{d,N}$ on both $d$ and $N$.

Before describing the contributions of our paper, we find it convenient to make a transformation of this problem.
For integer $n\ge 1$, we use the notation $\mathbf{x}=(x_1,\cdots, x_n)\in \RR^n$, 
$$
|\mathbf{x}|_n =\left(\sum_{k=1}^n x_k^2\right)^{1/2}.
$$
We map $\RR^d$ to the unit sphere:
$$
\SS^d=\{\mathbf{x}\in\RR^{d+1} : |\mathbf{x}|_{d+1}=1\}, 
$$
and its upper hemisphere:
$$
\SS^d_+=\{\x\in \SS^d : x_{d+1}>0\}.
$$
One of the reasons for this transformation is that the compactness of $\SS^d$ simplifies the analysis rather than using the compact open topology on $\RR^d$ directly.
We consider the mapping $\pi^* : \RR^d\to\SS^d_+$
 given by
\be\label{eq:euclid_to_sphere}
\pi^*(x_1,\cdots,x_d)= \left(\frac{x_1}{\sqrt{1+|\x|_d^2}},\cdots,\frac{x_d}{\sqrt{1+|\x|_d^2}}, \frac{1}{\sqrt{1+|\x|_d^2}}\right).
\ee
We note that 
\be\label{eq:sphere_to_euclid}
(\pi^*)^{-1}(u_1,\cdots,u_{d+1})=\left(\frac{u_1}{u_{d+1}}, \cdots, \frac{u_d}{u_{d+1}}\right).
\ee
The Frobenius norms of the Jacobians $d\pi^*$ and $d(\pi*)^{-1}$ are given by
\be\label{eq:frobenius}
\|d\pi^*\|_F^2= \frac{d+(d-1)|\x|_d^2}{(1+|\x|_d^2)^2}, \quad \|d(\pi*)^{-1}\|_F^2=\frac{d-1}{u_{d+1}^2}+\frac{|(u_1,\cdots,u_{d+1})|_{d+1}^2}{u_{d+1}^2}.
\ee
We note two consequences of \eqref{eq:frobenius}. 

First, it is elementary calculus to verify using \eqref{eq:frobenius}  that for any compact set $K\subset \RR^d$, there exist positive  constants $c_1(d,K)$, $c_2(d,K)$ such that
$$
c_1(d,K)|\pi^*(\x)-\pi^*(\y)|_{d+1}\le |\x-\y|_d\le c_2(d,K)|\pi^*(\x)-\pi^*(\y)|_{d+1}, \qquad \x,\y\in K.
$$
We note that the dependence on $d$ of $c_1(d,K)$, $c_2(d,K)$ is $\sim d^{1/2}$.

Second, $\pi^*$ is a diffeomorphism from $\RR^d$ to $\SS^d_+$. 
Continuous functions on $\RR^d$ vanishing at infinity are mapped to continuous, even functions on $\SS^d$, preserving smoothness. 
In particular, there is no loss of generality or smoothness of the functions involved if the operator $\mathcal{I}_{d;K_\mathfrak{X}}$ takes values in $\SS^d$ rather than $\RR^d$.
Thus, the problem of approximating a function of $\mathcal{I}_{d;K_\mathfrak{X}}(F)$ may be transformed into the problem of approximation of a function on $\SS^d$.

Similar transformations from a Euclidean ball to the sphere are well known in approximation theory (e.g., \cite{dai2013approximation, michel2012lectures}). We have only given one example of how such a transformation can be made in general.
As argued in \cite{dingxuanpap}, the transformation simplifies the study of neural networks with ReLU activation function by representing them as zonal function networks instead. 
(A zonal function network is a function of the form $\x\mapsto \sum_{k=1}^n a_k\phi(\w_k\cdot\x)$, $\w_k, \x\in\SS^d$, $a_k\in\RR$.)

For approximation on the sphere, spherical polynomials are the most natural class of approximants.
In several papers \cite{mnw2, zfquadpap, sphrelu}, we have described how spherical polynomials can be synthesized as zonal function networks with fixed weights independent of the function being approximated (target function). 
Our zonal function networks are linear operators on $C(\SS^d)$, and can be written as linear combination of pre-fabricated networks with coefficients given by the values of the target function at scattered data on the sphere.
This feature is extremely important in our current problem where a lot of functions need to be approximated. 
It ensures that the networks don't need to be trained separately for each function as is done, e.g., in \cite{lu2019deeponet}.
Given the close connection between zonal function networks and spherical polynomials already established in our  previous papers, we will focus in this  paper on approximation by spherical polynomials, mainly with the objective of dealing with the trade-off mentioned at the beginning of this sub-section.

\subsection{Contributions of this paper}\label{bhag:contributions}

We highlight some of the main contributions of this paper.
\begin{enumerate}
\item Our construction in the ``main processing'' step in Figure~\ref{fig:oper_approx} involves  pre-fabricated ``networks'' $\mathbb{G}_{d;N}$ using $N$ parameters in each sub-box of that box, so that the entire process involves $mN$ pre-computed parameters.
There is no training involved in the traditional sense.
The actual computation of $\mathcal{F}(F)$ then involves only $m$ matrix vector multiplications.
Thus, the total number of parameters is $(d+N)m$ rather than $dNM$. 
\item Unlike most other papers on the subject of approximation of operators, we provide error estimates in the supremum norm rather than a Hilbert space norm, and in fact, provide pointwise estimates in both deterministic and probabilistic sense (rather than in the sense of an expected value of a loss function). 
\item While approximation on the sphere is very well studied, the error bounds for approximation of smooth target functions typically involve unspecified constants depending upon the dimension of the sphere. 
The known examples where the constants can be computed explicitly involve constants that grow exponentially with the dimension (e.g., \cite{frankgreek, mhaskar2016localized}).
Part of the  problem is in the definition of the right smoothness classes.
We define a  smoothness class that is a natural generalization of the notion of Lipschitz continuity, and prove approximation results with constants of the form $d^{1/6}$. In \cite{newman1964jackson}, Newman and Shapiro have given bounds with constants that are independent of the dimension, but these hold only for Lipschitz continuous functions. 
We could not find a reference where similar bounds are achieved for functions with higher smoothness.
\item We achieve approximation using linear operators for which we give explicit constructions using either values of the target function at arbitrary locations on $\SS^d$ (scattered data, in contrast to designated locations such as the points at which special quadrature formulas such as Driscoll-Healy hold), or Fourier-Laplace coefficients of $f$.
\item We prove two kinds of local approximation results.
\begin{enumerate}
\item One is that the approximation of $f\in C(\SS^d)$ near a point $\x\in\SS^d$ involves only the values of $f$ in a small neighborhood of $\x$. 
This mitigates the effect of requiring a large number of functions in $K_\mathfrak{X}$ to obtain an approximation to $\mathcal{F}(F)$ even for a single function $F$.
\item The second aspect is to construct a globally defined operator for approximation with the property that the degree of approximation at any point adjusts \textbf{automatically} according to the smoothness of the target function at that point.
\end{enumerate}
\item The ideas in the paper can be extended in various ways.
\begin{enumerate}
\item Since $K_\mathfrak{X}$ is a compact set, for every $\epsilon>0$, there exists a finite cover of the set with finitely many balls of radius $<\epsilon$. 
Denoting the centers of this ball by $F_\ell$,
a possibly time consuming approximation of $\mathcal{F}(F_\ell)$ at the centers of each of these balls would give us a table so that $\mathcal{F}(F)$ can be computed using a table look-up, where we find $F_\ell$ close to $F$ and return the value of $\mathcal{F}(F_\ell)$ as an approximation to $\mathcal{F}(F)$.
\item It is possible to iterate this paradigm somewhat trivially to the case when one wants to approximate an operator acting on the space of operators, and so on.
\end{enumerate}
\end{enumerate}

\bhag{Preparatory material}\label{bhag:notation}
 In order to state our main results in Section~\ref{bhag:mainresults}, we need to describe some background on approximation on the sphere.
It is necessary for our proofs also to introduce some basic concepts about Jacobi polynomials. 
Section~\ref{bhag:jacobiprep} discusses some elementary facts about Jacobi polynomials.
In Section~\ref{bhag:sphpoly}, we introduce some basic facts about the sphere, including spherical polynomials and their connection with Jacobi polynomials.
An essential role in our theory is played by quadrature formulas on the sphere, which are discussed in Section~\ref{bhag:quadrature}. 
The kernels and operators which are used in our constructions are described in Section~\ref{bhag:kernels}.

 \subsection{Jacobi polynomials}\label{bhag:jacobiprep}
A standard reference for the material here is the book \cite{szego} of Szeg\"o. For $\alpha, \beta>-1$, $x\in (-1,1)$ and integer $\ell\ge 0$, the Jacobi polynomials $p_\ell^{(\alpha,\beta)}$ are defined by the Rodrigues' formula \cite[Formulas~(4.3.1), (4.3.4)]{szego}
\be\label{eq:rodrigues}
(1-x)^\alpha(1+x)^\beta p_\ell^{(\alpha,\beta)}(x)=\left\{\frac{2\ell+\alpha+\beta+1}{2^{\alpha+\beta+1}}\frac{\Gamma(\ell+1)\Gamma(\ell+\alpha+\beta+1)}{\Gamma(\ell+\alpha+1)\Gamma(\ell+\beta+1)}\right\}^{1/2}\frac{(-1)^\ell}{2^\ell \ell!}\frac{d^\ell}{dx^\ell}\left((1-x)^{\ell+\alpha}(1+x)^{\ell+\beta}\right).
\ee

Each $p_\ell^{(\alpha,\beta)}$ is a polynomial of degree $\ell$ with positive leading coefficient.  We have the orthogonality relation
\be\label{eq:jacobiortho}
\int_{-1}^1 p_\ell^{(\alpha,\beta)}(x)p_j^{(\alpha,\beta)}(x)(1-x)^\alpha(1+x)^\beta dx=\delta_{\ell,j},
\ee
and
\be\label{eq:pkat1}
\begin{aligned}
p_\ell^{(\alpha,\beta)}(1)&=\left\{\frac{2\ell+\alpha+\beta+1}{2^{\alpha+\beta+1}}\frac{\Gamma(\ell+1)\Gamma(\ell+\alpha+\beta+1)}{\Gamma(\ell+\alpha+1)\Gamma(\ell+\beta+1)}\right\}^{1/2}\frac{\Gamma(\ell+\alpha+1)}{\Gamma(\alpha+1)\Gamma(\ell+1)}\\
&=\left\{\frac{2\ell+\alpha+\beta+1}{2^{\alpha+\beta+1}}\frac{\Gamma(\ell+\alpha+1)\Gamma(\ell+\alpha+\beta+1)}{\Gamma(\ell+1)\Gamma(\ell+\beta+1)}\right\}^{1/2}\frac{1}{\Gamma(\alpha+1)} .
\end{aligned}
\ee
We have (\cite[Formula~(4.5.3), (4.3.4)]{szego}, \cite[Formula~(41), (42)]{mhaskar2016localized})
\be\label{eq:jacobikern}
K_n^{(\alpha,\beta)}(x)=\sum_{\ell=0}^{n-1} p_\ell^{(\alpha,\beta)}(1)p_\ell^{(\alpha,\beta)}(x)= \frac{2\alpha+2}{2n+\alpha+\beta}p_{n-1}^{(\alpha+1,\beta)}(1) p_{n-1}^{(\alpha+1,\beta)}(x),
\ee
and
\be\label{eq:jacobikernat1}
K_n^{(\alpha,\beta)}(1)=\sum_{\ell=0}^{n-1} p_\ell^{(\alpha,\beta)}(1)^2=\frac{1}{2^{\a+\beta+1}\Gamma(\a+1)\Gamma(\a+2)}\frac{\Gamma(n+\a+1)\Gamma(n+\a+\beta+1)}
{\Gamma(n)\Gamma(n+\beta)}
\ee
Clearly, for any polynomial $P$ of degree $<n$, 
\be\label{eq:jacobireprod}
\int_{-1}^1 P(x)K_n^{(\alpha,\beta)}(x)(1-x)^\alpha(1+x)^\beta dx =P(1),
\ee
and in particular,
\be\label{eq:jacobikernsq}
\int_{-1}^1 \left(K_n^{(\alpha,\beta)}(x)\right)^2(1-x)^\alpha(1+x)^\beta dx = K_n^{(\alpha,\beta)}(1).
\ee
 \subsection{Spherical polynomials}\label{bhag:sphpoly}
Most of the following information is based
on \cite{mullerbk}, \cite[Section~IV.2]{steinweissbk}, and \cite[Chapter XI]{batemanvol2}, although we use a
different notation.
 Let $d\ge 1$ be an integer, $\SS^d$ denote the unit sphere
\be\label{eq:unitsphere}
\SS^d=\{\x=(x_1,\cdots,x_{d+1}) : |\x|_{d+1}^2=x_1^2+\cdots+x_{d+1}^2 =1\}.
\ee
For $\delta\in (0,2]$, let
\be\label{eq:balldef}
\BB(\x,\delta)=\{\y\in\SS^d : |\x-\y|_{d+1}\le \delta\}.
\ee 
Let $\mu_d^*$ be the Riemannian volume measure on $\SS^d$, normalized so that $\mu_d^*(\SS^d)=1$.
We note that the volume of $\SS^d$ itself is given  by
\be\label{eq:sphvol}
\omega_d=\frac{2\pi^{(d+1)/2}}{\Gamma((d+1)/2)}=\frac{\sqrt{\pi}\Gamma(d/2)}{\Gamma((d+1)/2)}\omega_{d-1}=\begin{cases}
\disp\omega_{d-1}\int_{-1}^1 (1-x^2)^{d/2-1}dx, &\mbox{if $d\ge 2$,}\\[1ex]
2\pi, &\mbox{if $d=1$,}
\end{cases}
\ee
and the measure $\mu_d^*$ is defined recursively by
\be\label{eq:sphmeasuredef}
d\mu_d^*=
\begin{cases}
\disp\frac{\omega_{d-1}}{\omega_d}d\mu_{d-1}^*(\x')\sin^{d-1}\theta d\theta, \qquad \x=(\x'\sin\theta,\cos\theta)\in \SS^d, \quad \x'\in\SS^{d-1}, &\mbox{ if $d\ge 2$},\\[1.5ex]
\disp\frac{1}{2\pi} d\theta, \qquad \x=(\sin\theta, \cos\theta)\in\SS^1, &\mbox{ if $d=1$}.
\end{cases}
\ee
The measure $\mu_d^*$ is rotation invariant.  Therefore, it is easy to verify that if $f :[-1,1]\to\RR$ and $t\mapsto f(t)(1-t^2)^{d/2-1}$ is integrable with respect to the Lebesgue measure on $[-1,1]$,  then
\be\label{eq:zonalintegral}
\int_{\SS^d}f(\x\cdot\y)d\mu_d^*(\y)= \frac{\omega_{d-1}}{\omega_d}\int_{-1}^1 f(t)(1-t^2)^{d/2-1}(t)dt, \qquad \x\in\SS^d.
\ee

For a fixed integer $\ell\ge 0$, the restriction to $\SS^d$ of a
homogeneous harmonic polynomial of exact degree $\ell$ is called a spherical
harmonic of degree $\ell$. 
The class of all spherical harmonics of degree
$\ell$ will be denoted by $\HH^d_\ell$. The
spaces $\HH^d_\ell$ are mutually orthogonal relative to
the inner product of $L^2(\mu_d^*)$. 
An orthonormal basis for $\HH^d_\ell$ is $\{Y_{\ell,k}\}_{k=1,\cdots, \mathsf{dim}(\HH^d_\ell)}$.

One has the
well-known addition formula \cite{mullerbk} and \cite[Chapter XI, Theorem 4]{batemanvol2} connecting $Y_{\ell,k}$'s with Jacobi polynomials defined in \eref{eq:rodrigues}:
\begin{equation}
\label{eq:addformula}
 \sum_{k=1}^{\mathsf{dim}(\HH^d_\ell)} Y_{\ell,k}(\x)\overline{Y_{\ell,k}(\y)} =
\frac{\omega_d}{\omega_{d-1}} p_\ell^{(d/2-1,d/2-1)}(1)p_\ell^{(d/2-1,d/2-1)}(\x\cdot\y), \qquad
\ell=0,1,\cdots.
\end{equation}
For $n\ge 0$, we denote by $\Pi_n^d$ the set of restrictions to $\SS^d$ of all algebraic polynomials of degree $<n$.
In this definition, we allow $n$ to be a non-integer, so as to be able to write, for example, $\Pi_{n/2}^d$ rather than the more cumbersome $\Pi_{\lfloor n/2\rfloor}^d$.
For integers $n, \ell, d\ge 1$,  we have 
\be\label{eq:dimpin}
\mathsf{dim}(\Pi_n^d)=\begin{cases}
\disp(2n+d-2)\frac{\Gamma(n+d-1)}{\Gamma(n)\Gamma(d+1)} & \mbox{ if $n\ge 2$},\\[1ex]
1, & \mbox{ if $n=1$,}
\end{cases}
\qquad
\mathsf{dim}(\HH^d_\ell)=
\begin{cases}
\mathsf{dim}(\Pi_\ell^{d-1}) & \mbox{ if $d\ge 2$},\\[1ex]
2 &\mbox{ if $d=1$.}
\end{cases}
\ee
In view of \eqref{eq:addformula},
the reproducing kernel $K_{d;n}$ for $\Pi_n^d$ is defined using Jacobi polynomials by
\be\label{eq:sphreprodkern}
K_{d;n}(x)=\frac{\omega_d}{\omega_{d-1}}K_n^{(d/2-1,d/2-1)}(x)=\frac{2\sqrt{\pi}\Gamma((d+2)/2)}{\Gamma((d+1)/2)(2n+d-2)}p_{n-1}^{(d/2,d/2-1)}(1)p_{n-1}^{(d/2,d/2-1)}(x), \qquad x\in [-1,1].
\ee
Thus, we have the reproduction formula
\be\label{eq:sphreprod}
P(\x)=\int_{\SS^d}P(\y)K_{d;n}(\x\cdot\y)d\mu_d^*(\y), \qquad P\in \Pi_n^d.
\ee
Using \eqref{eq:sphreprod} with $K_{d;n}(\x\cdot\y)$ in place of $P$, and using \eqref{eq:addformula}, we deduce that
\be\label{eq:sphreprodl2}
\mathsf{dim}(\Pi_n^d)=K_{d;n}(1)=\int_{\SS^d}K_{d;n}(\x\cdot\y)^2d\mu_d^*(\y), \qquad \x\in\SS^d.
\ee

\subsection{Quadrature formula}\label{bhag:quadrature}
\begin{definition}\label{def:quadmeasure}
Let $n\ge 1$. A  measure $\nu$ on $\SS^d$ is called a \textbf{quadrature measure of order $n$} if for every $P\in\Pi_n^d$,
\be\label{eq:quadrature}
\int_{\SS^d} Pd\nu =\int_{\SS^d}Pd\mu^*. 
\ee
A measure $\nu$ is called a \textbf{Marcinkiewicz-Zygmund  measure of order $n$} (abbreviated by $\nu\in \mathsf{MZ}(d;n)$) if  for every $P\in\Pi_{n/2}^d$:
\be\label{eq:mzineq}
 \int_{\SS^d} |P|^2d|\nu| \le c\int_{\SS^d}|P|^2d\mu^*,
\ee
for some positive constant $c$. The infimum of all such constants will be denoted by $\tn\nu\tn_{d;n}$.
A measure $\nu$ is called a \textbf{Marcinkiewicz-Zygmund quadrature measure of order $n$} (abbreviated by $\nu\in \mathsf{MZQ}(d;n)$) if both \eqref{eq:quadrature} and \eqref{eq:mzineq} hold for every $P\in \Pi_n^d$.
\end{definition}

\begin{rem}\label{rem:reconciliate}
{\rm
This definition is essentially a special case of the definition given in Section~\ref{bhag:approx_manifold}. 
In \cite{modlpmz}, we have proved that the condition \eqref{eq:mzineq} is equivalent to the same condition with the $L^2$ norm replaced by the $L^1$ norm. 
However, the constants there depend upon the dimension of the manifold (sphere in this context). 
The use of $L^2$ norm, and requiring \eqref{eq:mzineq} to hold for $\Pi_{n/2}$ instead of $\Pi_n$ in the definition allows us to use positive quadrature formulas such as the one described in Theorem~\ref{theo:tchakalofftheo} below directly.
\qed}
\end{rem}
\begin{rem}\label{rem:monotonicity}
{\rm
It is clear that if $n<m$ then a quadrature measure (respectively, Marcinkiewicz-Zygmund measure, respectively, Marcinkiewicz-Zygmund quadrature measure) of order $m$ is also a quadrature measure (respectively, Marcinkiewicz-Zygmund measure, respectively, Marcinkiewicz-Zygmund quadrature measure) of order $n$, and $\tn \nu\tn_{d;n}\le \tn \nu\tn_{d;m}$.
\qed}
\end{rem}

\begin{rem}\label{rem:quadrature}
{\rm
In \cite{mnw1,modlpmz}, we have proved if $\C\subset\SS^d$ is any finite set, there exist  constant $c_1, c_2>0$ depending on $d$ with the following property: if
$$
\sup_{\x\in\SS^d}\min_{\y\in \C}|\x-\y|\le c_1/n,
$$
then there exists $\nu\in \mathsf{MZQ}(d;n)$ such that $\mathsf{supp}(\nu)$ is a subset of $\C$ containing at most $c_2\mathsf{dim}(\Pi_n^d)$ points. 
Existence of a positive measure $\nu\in\mathsf{MZQ}(d;n)$ is also proved in the same papers under the same conditions except for a smaller value of $c_1$.
In 
 \cite{mhaskar2020kernel}, we have estimated the cardinality of a random sample $\C$ that allows the condition mentioned above to be within a logarithmic multiple of $\mathsf{dim}(\Pi_n^d)$. 
 The various constants in all these constructions depend upon $d$ in an unspecified manner.
\qed}
\end{rem}

\begin{rem}\label{rem:positivequadrature}
{\rm
It is clear that if $\nu$ is a positive measure satisfying \eqref{eq:quadrature} for all $P\in\Pi_n^d$, then it satisfies \eqref{eq:mzineq} automatically for $P\in \Pi_{n/2}^d$, so that $\nu\in \mathsf{MZQ}(d;n)$ with $\tn\nu\tn_{d;n}=1$. 
\qed}
\end{rem}

\begin{rem}\label{rem:tchakaloff}
{\rm
We note the following theorem, called Tchakaloff's theorem \cite[Exercise~2.5.8, p.~100]{rivlin1974chebyshev}, that asserts in particular the existence of a positive quadrature formula satisfying \eqref{eq:quadrature} for all $P\in\Pi_n^d$ based on exactly $\mathsf{dim}(\Pi_n^d)$ points.
\begin{theorem}\label{theo:tchakalofftheo}
Let $\mathbb{X}$ be a compact topological space, $\{\phi_j\}_{j=0}^{N-1}$ be continuous real valued functions on $\mathbb{X}$, $\phi_0\equiv 1$, and $\mu^*$ be a probability measure on $\mathbb{X}$ (i.e., $\mu^*$ is a positive Borel measure with $\mu^*(\mathbb{X})=1$). Then there exist $N$ points $x_1,\cdots,x_N$ in $\XX$, and non--negative numbers $w_1,\cdots,w_N$ such that
\begin{equation}\label{tchakaloffquad}
 \sum_{k=1}^N w_k\phi_j(x_k)=\int_\mathbb{X} \phi_j(x)d\mu^*(x), \qquad j=0,\cdots,N-1.
\end{equation}
\end{theorem} 
An optimization procedure to compute the nodes and weights in this theorem is suggested in \cite{sloanfest}.
\qed}
\end{rem}

\subsection{Kernels and operators}\label{bhag:kernels}
Let $r\ge 0$. In Section~\ref{bhag:mainresults}, $r$ will be a parameter that defines the smoothness of the target function.
 In this paper, we will use heavily the following kernels, motivated by \cite{frankgreek, butzer, mhaskar2016localized} :
\be\label{eq:modifiedjacksonkern}
\widetilde{\Phi}_{d;n,r}(x)= K_{d;(d+2)n}(x)\frac{p_{dn}^{(d/2+r,d/2-2)}(x)}{p_{dn}^{(d/2+r,d/2-2)}(1)},
\ee
and
\be\label{eq:jacksonkerndef}
\Phi_{d;n,r}(x)=\widetilde{\Phi}_{d;n,r}(x)\left(\frac{1+x}{2}\right)^{n}= K_{d;(d+2)n}(x)\frac{p_{dn}^{(d/2+r,d/2-2)}(x)}{p_{dn}^{(d/2+r,d/2-2)}(1)}\left(\frac{1+x}{2}\right)^n.
\ee
Corresponding to the two kernels, we define two operators as follows.
If $\nu$ is a measure on $\SS^d$ having bounded total variation and $f$ is integrable with respect to $\nu$, we define
\be\label{eq:modifiedoperatordef}
\widetilde{\sigma}_{d;n,r}(\nu, f)(\x)=\int_{\SS^d}f(\y)\widetilde{\Phi}_{d;n,r}(\x\cdot\y)d\nu(\y), \qquad n >0, \ \x\in\SS^d,
\ee
and
\be\label{eq:sphoperatordef}
\sigma_{d;n,r}(\nu, f)(\x)=\int_{\SS^d}f(\y)\Phi_{d;n,r}(\x\cdot\y)d\nu(\y), \qquad n >0, \ \x\in\SS^d.
\ee

The operators defined above provide good approximation in the sense which we now describe.\\

\noindent\textbf{Constant convention}\\

\noindent\textit{In the sequel, the notation $A\ls B$ will denote $A\le cB$ for a positive constant $c$ that may depend upon fixed parameters under discussion, such as the smoothness parameter $r$ to be introduced in Section~\ref{bhag:mainresults}, but independent of $d$, $n$, $f$, or the points on the sphere. The notation $A\gs B$ will mean $B\ls A$, and $A\sim B$ will mean $A\ls B\ls A$. The notation $A=B+\O(C)$ will mean $|A-B|\ls C$. \qed}\\[1ex]

For $f\in C(\SS^d)$, we denote
\be\label{eq:degapproxdef}
E_{d;n}(f)=\min_{P\in\Pi_n^d}\|f-P\|_\infty.
\ee
\begin{theorem}\label{theo:goodsphapprox}
Let $d\ge 3$, $r\ge 0$, $n\ge 2(d+1)$,  and $f\in C(\SS^d)$.\\[1ex]
{\rm (a)} If $n, r>0$,  $P\in \Pi_n$, and $\nu$ is a quadrature measure of order $2(d+2)n$, then $\sigma_{d;n,r}(\nu,P)=P$, $\widetilde{\sigma}_{d;n,r}(P)=P$.
{\rm (b)} If $\nu\in\mathsf{MZQ}(d;2(d+2)n)$, then
\be\label{eq:degapproxstrong}
E_{d;2(d+2)n}(f)\le \|f-\widetilde{\sigma}_{d;n,r}(\nu,f)\|_\infty \ls d^{1/6}\tn\nu\tn_{d;2(d+2)n}E_{d;n}(f),
\ee 
and
\be\label{eq:degapproxstrongbis}
E_{d;2(d+2)n}(f)\le \|f-\sigma_{d;n,r}(\nu,f)\|_\infty \ls d^{1/6}\tn\nu\tn_{d;2(d+2)n}E_{d;n}(f).
\ee 
\end{theorem}

\bhag{Main results}\label{bhag:mainresults}
In this section, we describe our main results on the degree of approximation on $\SS^d$. 
We recall from Section~\ref{bhag:formulation} that each point on $\SS^d$ is potentially $\pi^*(\mathcal{I}_{d,K_\mathfrak{X}}(F))$ (cf. \eqref{eq:euclid_to_sphere}) for some $F\in K_\mathfrak{X}$.
So, for example, when we discuss points in a neighborhood of $\x=\pi^*(\mathcal{I}_{d,K_\mathfrak{X}}(F))\in\SS^d$, it is understood that the discussion refers to functions in a neighborhood of $F$.

Our first theorem deals with local approximation of smooth functions.
There are many definitions of smoothness of a function on the sphere (e.g., \cite{dai2013approximation}). Unlike the moduli of smoothness defined in the cited book, our definition (motivated by \cite[Chapter~VI, Section~2.3]{stein2016singular}) is coordinate free. We find it also more natural, and it leads to the right constants in our theorem below.

We recall that a function $f\in C(\SS^d)$ satisfies a H\"older condition of order $r\in (0,1]$ if 
$$
|f(\x)-f(\y)|\le c(f)|\x-\y|_{d+1}^r, \qquad \x,\y\in \SS^d.
$$
A local smoothness in this sense at $\x\in\SS^d$ would require the above estimate for all $\y$ in a neighborhood of $\x$. 
Fixing $\x$ one can think of $f(\x)\in\Pi_1^d$. 
Thus, one can say that $f$ is locally H\"older at $\x$ if
$$
\min_{P\in\Pi_1^d}\max_{\y\in \mathbb{B}(\x,\delta)}\frac{|f(\y)-P(\y)|}{|\x-\y|_{d+1}^r} <\infty.
$$
These considerations motivate the following Definition~\ref{def:smoothness}.

\begin{definition}\label{def:smoothness}
Let $f\in C(\SS^d)$, $r>0$ and $\x\in\SS^d$. The function $f$ is said to be $r$-smooth at $\x$ if there exists $\delta=\delta(d;f,\x)>0$  such that
\be\label{eq:smoothnessdef}
\|f\|_{d;r,\x}:=\|f\|_\infty+\min_{P\in \Pi_r^d}\max_{\y\in \mathbb{B}(\x,\delta)}\frac{|f(\y)-P(\y)|}{|\x-\y|_{d+1}^r} <\infty.
\ee
The class of all $f\in C(\SS^d)$ for which $\|f\|_{d;r,\x}<\infty$ will be denoted by $W_{d;r,\x}$.
The class $W_{d;r}$ will denote the set of all $f\in C(\SS^d)$ for which 
\be\label{eq:globalsmoothessdef}
\|f\|_{d;r}=\sup_{\x\in \SS^d}\|f\|_{d;r,\x} <\infty.
\ee
We note that for $f\in W_{d;r}$, we may choose $\delta(d;f)$ in \eqref{eq:smoothnessdef} to be independent of $\x$.
\end{definition}

\begin{rem}\label{rem:smoothexam}
{\rm
If $r$ is an integer and $f$ is $r$-times differentiable in a neighborhood of $\x$, then the function $\z\mapsto f(\z/|\z|_{d+1})$, $\z\in\RR^{d+1}$ is also $r$-times differentiable in a Euclidean neighborhood of $\x$. 
The restriction to $\SS^d$ of a Taylor polynomial of this function on the Euclidean neighborhood works as one of the polynomials in the definition \eqref{eq:smoothnessdef}.
\qed}
\end{rem}

\begin{rem}\label{rem:juergen}
{\rm
In the definition of local smoothness of $f$ at $\x$, it is tempting to let $\delta$ be independent of $\x$ by noting (cf. Lemma~\ref{lemma:sphwalsh}) that
$$
\|f\|_{d;r,\x}\le \|f\|_\infty +\min_{P\in\Pi_r}\max_{y\in\SS^d}\frac{|f(\y)-P(\y)|}{|\x-\y|_{d+1}^r}\ls \delta_\x^{-r} \|f\|_{d;r,\x}.
$$
In Theorem~\ref{theo:locsphapprox}, we wish to allow $n$ to be dependent on $\delta_\x$ (in  particular the smallest $n$ that satisfies all the conditions of that theorem).
Therefore, the factor $\delta_\x^{-r}$ may destroy the degree of approximation if we use in the definition the whole sphere rather than a ball of radius dependent on $\x$. 
Moreover, we feel that our definition underlines the role of local smoothness better than using the whole sphere in this definition.
\qed
}
\end{rem}

Our first theorem gives the local approximation properties of the operators for smooth functions.

\begin{theorem}\label{theo:locsphapprox}
Let $d\ge 3$,   $\x\in\SS^d$, $r=r(\x)>0$,and $f\in W_{d;r,\x}$. Let $nd\ge (d+r+1)^2$,  $\nu\in\mathsf{MZQ}(d;2(d+2)n)$.
If $n$ is large enough so that
 \be\label{eq:near_ball_def}
\delta_n=\sqrt{\frac{16r\log n}{n}}\le \delta(d;f,\x),
\ee
then
\be\label{eq:loc_deg_wholeint}
|f(\x)-\sigma_{d;n,r}(\nu,f)(\x)|\ls \frac{d^{1/6}}{\mathsf{dim}(\Pi_{2(2+d)n}^d)^{r/d}}\|f\|_{W_{d;r,\x}}\tn\nu\tn_{d;2(d+2)n}.
\ee
Moreover,
 \be\label{eq:loc_deg_approx}
 \left|f(\x)-\int_{\BB(\x,\delta_n)}\Phi_{d;n,r}(\x\cdot\y)f(\y)d\nu(\y)\right|\ls \frac{d^{1/6}}{\mathsf{dim}(\Pi_{2(d+2)n}^d)^{r/d}}\|f\|_{W_{d;r,\x}}\tn\nu\tn_{d;2(d+2)n}.
 \ee
 All the constants involved in $\ls$ depend only upon $r(\x)$ but are otherwise independent of $f$, $\x$, $n$, and $d$.
\end{theorem}
\begin{rem}\label{rem:local_approx_version}
{\rm
The estimate \eqref{eq:loc_deg_approx} shows that for any $\x$, the approximation of $f(\x)$ is accomplished using only  values of $f$ in a small neighborhood of $\x$.
In the case when $f\in W_{d;r(\x),\x}$ for every $\x\in\SS^d$ for some $r(\x)>0$, then the quantity $\delta(d;f,\x)$ may be chosen independent of $\x$. The estimate \eqref{eq:loc_deg_wholeint} then shows that at each point $\x\in\SS^d$, the error $|f(\x)-\sigma_{d;n,r}(\nu,f)(\x)|$ adjusts automatically to the smoothness of $f$ at $\x$, even though the operator is defined in a global manner without requiring any smoothness on $f$ at all.
\qed}
\end{rem}

\begin{rem}\label{rem:r_vs_s}
{\rm
In the statement of Theorem~\ref{theo:locsphapprox}, the operator seems to depend upon the smoothness of $f$ near $\x$. 
It will be clear from the proof that one does not need to know this smoothness beforehand; any $r$ greater than the actual smoothness will work.
We chose to write the theorem in this way to avoid unnecessarily complicated notation that does not add much insight.
\qed}
\end{rem}

The following theorem is a global version of Theorem~\ref{theo:locsphapprox}.
\begin{theorem}\label{theo:globalsphapprox}
Let $d\ge 3$, $n\ge 2(d+1)$, and $f\in C(\SS^d)$.
If $r>0$, $f\in W_{d;r}$, and $n$ is large enough so that
$dn\ge (d+r+1)^2$ and\eqref{eq:near_ball_def} is satisfied with $\delta(d;f)$ in place of $\delta(d;f,\x)$,
then
\be\label{eq:sphsobolapprox}
E_{d;2(d+2)n}(f)\ls \frac{d^{1/6}}{\mathsf{dim}(\Pi_{2(d+2)n}^d)^{r/d}}\|f\|_{W_{d;r}}.
\ee
\end{theorem}

\begin{rem}\label{rem:global_optimality}
{\rm
Theorem~\ref{theo:tchakalofftheo} shows the existence of a positive measure $\nu\in \mathsf{MZQ}(d;2(d+2)n)$ supported on exactly $\mathsf{dim}(\Pi_{2(d+2)n}^d)$ points. 
Theorems~\ref{theo:goodsphapprox} and \ref{theo:globalsphapprox} show that the operator $\widetilde{\sigma}_{d;n,r}$ (and also $\sigma_{d;n,r}$) provide optimal approximation in the sense of nonlinear widths based on the minimal number of samples of the target function.
\qed}
\end{rem}

Next, we discuss approximation of analytic functions. 
The following definition is motivated by a theorem of Siciak \cite{siciak1962some} regarding approximation of functions on a unit cube which are analytic in a so-called polyellipse.

\begin{definition}\label{def:analytic}
Let $f\in C(\SS^d)$, $\x\in\SS^d$, $\rho >0$. The function $f$ is said to be $\rho$-analytic at $x$ if there is exists $\delta=\delta(d;f,\x)>0$ such that 
\be\label{eq:loc_anal_norm}
\|f\|_{A_{d;\rho,\x}}=\|f\|_\infty +\sup_{n\ge 0}\left\{\exp\left(\rho n\right)\min_{P\in\Pi_n^d}\|f-P\|_{\infty,\BB(\x,\delta)}\right\}<\infty.
\ee
The class of all $f$ for which $\|f\|_{A_{d;\rho,\x}}<\infty$ will be denoted by $A_{d;\rho,\x}$. The class $A_{d;\rho}$ will denote the set of all $f\in C(\SS^d)$ for which
\be\label{eq:global_anal_norm}
\|f\|_{A_{d;\rho}}=\|f\|_\infty +\sup_{n>0}\left\{\exp\left(\rho n\right)E_{d;n}(f)\right\}<\infty.
\ee
We note that if $f\in A_{d;\rho}$ then we may choose $\delta(d;f,\x)$ in \eqref{eq:loc_anal_norm} to be independent of $\x$.
\end{definition}

The following theorem describes the analogue of Theorem~\ref{theo:locsphapprox} for locally analytic functions.
The global version is immediate from Theorem~\ref{theo:goodsphapprox} and the definitions.

\begin{theorem}\label{theo:analapprox}
Let $d\ge 3$, $r\ge 0$,  $nd\ge (d+r+1)^2$, $\nu\in\mathsf{MZQ}(d;2(d+2)n)$.
 If $\x\in\SS^d$, $f\in A_{d;\rho,\x}$, and $\delta=\delta(d;f,\x)$ be as in Definition~\ref{def:analytic}. 
Then
\be\label{eq:loc_anal_approx}
|f(\x)-\sigma_{d;n,r}(\nu,f)(\x)|\ls d^{1/6}\exp\left(-\min(\rho, \delta^2/4-2\log(4/\delta))n\right)\|f\|_{A_{d;\rho,\x}}\tn\nu\tn_{2(d+2)n},
\ee
and
\be\label{eq:loc_anal_approx_bis}
\left|f(\x)-\int_{\BB(\x,\delta(d;f,\x))}\Phi_{d;n,r}(\x\cdot\y)f(\y)d\nu_n(\y)\right| \ls d^{1/6}\exp\left(-\min(\rho,  \delta^2/4-2\log(4/\delta))n\right)\|f\|_{A_{d;\rho,\x}}\tn\nu\tn_{2(d+2)n}.
\ee
\end{theorem}

In applications to uncertainty quantification, the samples of the input functions $F$ (which correspond to the points on $\SS^d$ at which the target function $f$ is sampled) are taken from a probability distribution.
Our next theorem describes the approximation properties in this set up.
We consider only global approximation of smooth functions.
A result for approximation of analytic functions can be proved similarly.
The local  approximation would involve two options.
One can assume random samples from a distribution supported on a neighborhood of the point at which local approximation is considered. 
This case is covered by taking the function $f_0$ in the following theorem to be supported on  this neighborhood.
The other option is to take the random sample on the entire sphere, but use only those which lie in a small neighborhood of the point in question.
This would involve the  use of multiplicative Chernoff bounds.
This line of work will take us too far away from the main theme of this paper.

\begin{theorem}\label{theo:probtheo}
Let $r>0$, $d\ge 3$, $\epsilon>0$.
Let $M\ge 2$, and $\C=\{\y_1,\cdots,\y_M\}\subset \SS^d$ be random samples drawn from a probability  distribution $d\tau=f_0d\mu_d^*$ for some $f_0\in C(\SS^d)$. Let $f\in C(\SS^d)$, $r>0$, $f_0f\in W_{d;r}$, $n$ be large enough to satisfy the conditions of Theorem~\ref{theo:globalsphapprox}(b). 
We define
\be\label{eq:probestimator}
\widehat{\sigma}_{d;n,r}(\C,f)(x)=\frac{1}{M}\sum_{j=1}^M f(\y_j)\Phi_{d;n,r}(\x\cdot\y_j), \qquad \x\in\SS^d.
\ee
With $D_n=\mathsf{dim}(\Pi_{2(d+2)n}^d)=K_{d;2(d+2)n}(1)$, if
\be\label{eq:Meqn}
M\gs 2^{-d} d^{-1/3}\frac{\|f\|_\infty}{\|f_0f\|_\infty}D_n^{(2r+d)/d}\left(\log D_n+(d+1)\log d +d\log (16e/\pi)+\log(1/\epsilon)\right),
\ee
with an appropriate constant involved in the inequality,
then with $\tau$-probability $>1-\epsilon$,
\be\label{eq:probest}
\|f_0f-\widehat{\sigma}_{d;n,r}(\C,f)\|_\infty \ls \frac{d^{1/6}}{D_n^{r/d}}\|f_0f\|_{W_{d;r}}
\ee
Equivalently, with 
$$
B=(2r+d)\left\{\log d+\log(32e^2/\pi)-\frac{1}{2d}\log (d+1)\right\}, \qquad \tilde{M}=\frac{(2r+d)2^d}{d^{2/3}}M,
$$
and
\be\label{eq:ncond}
n\sim \left\{\frac{\tilde{M}}{\log \tilde{M}+B}\right\}^{1/(2r+d)}\sim \left\{\frac{M}{\log M +(2r+d)\log d}\right\}^{1/(2r+d)}
\ee
we have
\be\label{eq:probestbis}
\|f_0f-\widehat{\sigma}_{d;n,r}(\C,f)\|_\infty \ls d^{1/6}
\left\{\frac{M}{\log M+(2r+d)\log d}\right\}^{-r/(2r+d)}.
\ee
\end{theorem}

\bhag{An example}\label{bhag:manifold}
In this section, we illustrate how to apply the theory in order to approximate an operator $\mathcal{F}:C(\XX_1)\to C(\XX_2)$, where $\XX_1$, $\XX_2$ are smooth compact manifolds.
This requires a considerable background on the theory of function approximation on manifolds, which we summarize in Section~\ref{bhag:approx_manifold}. 
The details for the approximation of the operator itself are given in Section~\ref{bhag:manifold_algorithm}.
\subsection{Background}\label{bhag:approx_manifold}

Let $\XX$ be a compact, smooth, orientable manifold, with the geodesic metric $\rho$, and the Riemannian volume measure $\mu^*$, normalized to be a probability measure. 
Let $\{\phi_k\}$ be the  basis of $L^2(\mu^*)$ comprising eigenfunctions of the Laplace--Beltrami operator on $\XX$, orthonormal with respect to $\mu^*$, with each $\phi_k$ corresponding to the eigenvalue $-\lambda_k^2$. 
We assume that $\phi_0\equiv 1$, $\lambda_0=0$, and that $\lambda_k\uparrow\infty$ as $k\to\infty$.
For $n>0$, we write $\Pi_n=\mathsf{span}\{\phi_k : \lambda_k<n\}$, and assume that $\bigcup_{n>0}\Pi_n$ is dense in $C(\XX)$. 
We assume further the \textit{strong product assumption}: there exists $A^*\ge 2$ such that for any $P,Q\in\Pi_n$, the product $PQ\in \Pi_{A^*n}$.
We encode all these assumptions by stating that the quintuple  $\Xi=(\XX,\rho,\mu^*,\{\phi_k\}, \{\lambda_k\})$ is a \textbf{\textit{system}}.
We will mention $\Xi$ in the notation whenever it is necessary to prescribe the system to avoid confusion; for example, we will write $\Pi_n(\Xi)$ for $\Pi_n$, $E_n(\Xi;f)$ for $E_n(f)$ below, etc.
In this appendix, $\ls,\gs, \sim$ will involve constants that depend upon the system $\Xi$.
We note that if the dimension of $\XX$ is $q$, then $\mathsf{dim}(\Pi_n)\sim n^q$.

A a signed (or positive) measure $\nu$ on $\XX$ with bounded total variation is called a Marcinkiewicz-Zygmund  quadrature measure of order $n$ (i.e., $\nu\in \mathsf{MZQ}(n)$) if both of the following conditions are satisfied for all $P\in\Pi_{A^*n}$:
\be\label{eq:manifoldmzq}
\int_{\XX} Pd\nu=\int_{\XX} Pd\mu^*, \qquad \int_{\XX} |P|d\nu\le \tn\nu\tn_{n}\int_{\XX} |P|d\mu^*
\ee
for a (minimal) positive constant $\tn\nu\tn_{n}$.
In \cite{frankbern,modlpmz}, we have proved that there exists a constant $B>0$ with the following property: If
$\C\subset\XX$ (generally, a finite subset) satisfies
$$
\sup_{x\in\XX}\min_{y\in \C}\rho(x,y)\le B/n,
$$
then there exists  $\nu\in \mathsf{MZQ}(n)$ with $\mathsf{supp}(\nu)\subset \C$. When $\C$ is finite, we may choose $|\mathsf{supp}(\nu)|\sim \mathsf{dim}(\Pi_{A^*n})\sim n^q$.

For $f\in C(\XX)$, we write 
$$
E_{n}(f)=\min_{P\in \Pi_n}\|f-P\|_{\infty;\XX}.
$$

If $\nu$ is a signed (or positive) measure on $\XX_j$ with bounded total variation, we define
$$
\hat{f}(\nu,k)=\int_{\XX}f\phi_kd\nu.
$$ 
If $f\in C(\XX)$, and $r>0$, we say that $f\in \mathbb{W}_{r}$ if there exists $\mathcal{D}_r(f)\in C(\XX)$ such that $\widehat{\mathcal{D}_r(f)}(\mu^*,k)=\lambda_k^r\hat{f}(\mu^*,k)$, $k\in \mathbb{N}_0$, and write
\be\label{eq:example_sobol}
\|f\|_{W_{r}}=\|f\|_{\infty,\XX}+\|\mathcal{D}_r(f)\|_{\infty,\XX}.
\ee
Then the unit ball $K_{r}=\{f\in W_{r} : \|f\|_{W_{r}}\le 1\}$ is a compact subset of $C(\XX)$.

Let $H :\RR\to [0,1]$ be infinitely differentiable function with $H(t)=1$ for $0\le t\le 1/2$ and $H(t)=0$ for $t\ge 1$.
We define
$$
\Phi_{n}(x,y)=\sum_{k=0}^\infty H\left(\frac{\lambda_k}{n}\right)\phi_k(x)\phi_k(y), \qquad x, y\in\XX,
$$
and define
\be\label{eq:example_sigma}
\sigma_n(\nu,f)=\int_{\XX}f(y)\Phi_{n}(x,y)d\nu(y).
\ee
We have proved in \cite{mauropap, heatkernframe, compbio} that if $\nu\in \mathsf{MZQ}(A^*n)$ then
\be\label{eq:manifolddegapprox}
\|f-\sigma_n(\nu,f)\|_{\infty,\XX}\ls n^{-r}\sim (\mathsf{dim}(\Pi_n))^{-r/q}, \qquad f\in K_r.
\ee

Let $\C\subset \XX$ be a finite subset, and
$$
\eta=\min_{x,y\in \C,x\not=y}\rho(x,y).
$$
 It can  be shown  (see \cite[Theorem~6.1]{mason2021manifold}) for a recent proof) using the localization properties of the kernels $\Phi_n$ proved in \cite{mauropap} that the following statement holds. Let $\mathbf{v}=(v_x)_{x\in\C}\in\RR^{|\C|}$. For $n\gs \eta^{-1}$, the following system of equations has a unique solution $(b_z)_{z\in\C}$:
\be\label{eq:interp_eqn}
\sum_{z\in \C}b_z\Phi_n(x,z)=v_x, \qquad x\in \C.
\ee
We will denote the function $x\mapsto \sum_{z\in \C}b_z\Phi_n(x,z)$ by $F_{\mathbf{v}}$.
Clearly, the function $F_{\mathbf{v}}\in W_r$ for every $r>0$. 

\subsection{Approximation of the operator}\label{bhag:manifold_algorithm}
We illustrate the application of our main results in the case of approximation of operators defined on functions on a manifold. 
In the sense of the Section~\ref{bhag:approx_manifold}, we assume two systems $\Xi_j=(\XX_j, \rho_j,\mu_j^*, \{\phi_{k,j}\}, \{\lambda_{k,j}\})$, $j=1,2$.
The constants involved in $\ls$, $\gs$, $\sim$ may depend upon both the systems.
In this case, the operator $\mathcal{F} : C(\XX_1)\to C(\XX_2)$, and it is desired that the approximation of this operator should take place on the unit ball $K_{\mathfrak{X}}$ of $W_{s_1}(\Xi_1)\subset C(\XX_1)$.
We assume that the image of $K_{\mathfrak{X}}$ under $\mathcal{F}$, denoted by $\mathcal{K}_{\mathfrak{Y}}$ is a (necessarily compact in $C(\XX_2)$)  subset of $W_{s_2}(\Xi_2)\subset C(\XX_2)$.
For $j=1,2$, we consider finite subsets $\C_j\subset \XX_j$,  such that there are measures $\nu_j\in MZQ(\Xi_j,A^*(\Xi_j)n_j)$ supported on $\C_j$. 
Necessarily $|\C_j|=m_j\sim n_j^{q_j}$.
For $j=1, 2$, we consider the information operators (encoders) $\mathcal{I}_j : C(\XX_j)\to \RR^{m_j}$ given by $\mathcal{I}_j(f)=(f(x))_{x\in\C_j}$, and the 
the corresponding reconstruction operators 
 given by $\mathcal{A}_j=\sigma_{n_j}(\Xi_j;\nu_j;\circ)$. 
 Then (cf. \eqref{eq:manifolddegapprox})
\be\label{eq:sobolbds}
 \|\mathcal{A}_j(\mathcal{I}_j(f))-f\|_{C(\XX_j)}\ls m_j^{-s_j/q_j}\|f\|_{W_{s_j}(\Xi_j)}, \qquad f\in W_{s_j}(\Xi_j),\  j=1,2. 
 \ee
The question of approximating $\mathcal{F}(F)$ is now reduced to approximating the function from $\RR^{m_1}\to \RR^{m_2}$ whose value at $\mathcal{I}_j(F)$ is $\mathcal{I}_2(\mathcal{F})(F)$. 
After the transformation to the sphere, the method suggested in Section~\ref{bhag:mainresults} now yields the degree of approximation for this function.
We summarize the application of our theory in the form of an algorithm (Figure~\ref{fig:algorithm}).

\begin{figure}[ht]

\noindent\hrulefill\\
\textbf{Pre-computation}
\begin{itemize}
\item Find a finitely supported quadrature measure $\nu\in MZ(d;2(d+2)n)$, $\mathsf{supp}(\nu)=\{\y_1,\cdots,\y_M\}\subset \SS^d$, $w_k=\nu(\{\y_k\})$.
\item Find $f_k\in W_{s_1}(\Xi_1)$ $,k=1,\cdots, M$ such that 
$$
\y_k=\pi^*\left((f_k(x))_{x\in\C_1}\right).
$$
\item Compute 
$$
\z_k=\left(\mathcal{F}(f_k)(z)\right)_{z\in \C_2},
$$

\end{itemize}

\noindent\textbf{Approximation of the operator}\\

Given $F$,
\begin{itemize}
\item Compute
$$
\x=\pi^*\left((F(x))_{x\in\C_1}\right).
$$
\item Compute
$$
\mathbf{b}=\sum_{k=1}^M w_k\z_k\Phi_{d;n,r}(\x\cdot\y_k).
$$
\item Return 
$$
\mathcal{A}_2(\mathbf{b}).
$$
\end{itemize}
\hrulefill
\caption{Algorithm for approximation of $\mathcal{F}:C(\XX_1)\to C(\XX_2)$.}
\label{fig:algorithm}
\end{figure}

In this section, we write $G_\y=F_{(\pi*)^{-1}(\y)}$ for $\y\in\SS^q$. 
We may view the operator $\sigma_{m_1;n,r}(\Xi_1;\nu,\circ)$ in \eqref{eq:sphoperatordef} as a pre-fabricated network using the functions $\{G_\y : \y\in \mathsf{supp}(\nu_1)\}$ as training data, and $f(\y)$ replaced by each of the components  $\mathcal{F}(\sigma_{n_1}(\Xi_1;\nu_1, G_{\y}))(z)$ for $\y\in \mathsf{supp}(\nu_1)$, $z\in\C_2$ in turn.
The total number of parameters in this process is $Nm$, where $N=\mathsf{dim}(\Pi_{2(d+2)n}^d)$. 
All these parameters are pre-computed and fixed for approximation of $\mathcal{F}(F)$ for any $F\in W_{s_1}(\Xi_1)$.
After this, for any $F\in W_{s_1}(\Xi_1)$, the approximation of $\mathcal{F}(F)$ involves only $dm$ further parameters.
Thus, the overall complexity of main processing (cf. Figure~\ref{fig:oper_approx}) in terms of parameters, including pre-computation, is $\O((d+N)m)$ rather than $\O(dNm)$.
For local approximation at $F\in W_{s_1}(\Xi_1)$, we need  to use only those $G_\y$'s out of the pre-computed functions that are close to $F$ as indicated in Theorem~\ref{theo:locsphapprox}, resulting in a further reduction of complexity.

\bhag{Proofs}\label{bhag:proofs}
The main purpose of this section is to prove all the  theorems in Section~\ref{bhag:mainresults}. 
As expected, a major part of the proofs is to obtain estimates on the Lebesgue constant for $\Phi_{d;n,r}$. 
It is convenient to summarize some of the required calculations in the context of Jacobi polynomials in general. 
Accordingly, we prove in Section~\ref{bhag:jacobi} some estimates on some integrals involving Jacobi polynomials.
The estimates on the Lebesgue constants and some other technical estimates about spherical polynomials are proved in Section~\ref{bhag:sphereest}. 
The proofs of the theorems in Sections~\ref{bhag:sphpoly} and \ref{bhag:mainresults} are given in Section~\ref{bhag:mainproofs}.

\subsection{Estimates on Jacobi polynomials}\label{bhag:jacobi}
In this sub-section only, the notation $A\ls B$ will denote $A\le cB$ for an \textbf{absolute} positive constant, $A\gs B$ will mean $B\ls A$, and $A\sim B$ will mean $A\ls B\ls A$. The notation $A=B+\O(C)$ will mean $|A-B|\ls C$. 

The following theorem summarizes some of the important estimates we will use in this paper.

\begin{theorem}\label{theo:jacobisq}
{\rm (a)} Let $a\ge b\ge 0$, $A\ge B\ge 1$, $0<p<\infty$, $2b\ge pB+p/2$, and $2a+2<pA+p/2$. Further assume that
\be\label{eq:jacbbincond}
n\gs A\left\{\frac{2a+2}{p(A+1/2)-2a-2}\right\}^{1/p}.
\ee
We have
\be\label{eq:jacobilpintegral}
\begin{aligned}
\int_{-1}^1&\left|\frac{p_n^{(A,B)}(x)}{p_n^{(A,B)}(1)}\right|^p(1-x)^a(1+x)^bdx\\
& \ls  \frac{2^{a+b}}{2a+2}\left\{\frac{A^{1/6}}{2^{(A+B-1)/2}p_n^{(A,B)}(1)}\right\}^{(2a+2)/(A+1/2)}\left\{\frac{2a+2}{p(A+1/2)-2a-2}\right\}^{(2a+2)/p(A+1/2)}
\end{aligned}
\ee
{\rm (b)} If $\alpha\ge\beta \ge 2$, $r\ge s>-1$, $n\ge (2\alpha+r+1)^2$, then
\be\label{eq:jacobisqintegral}
\begin{aligned}
\int_{-1}^1& \left|\frac{p_n^{(\alpha+r+1,\beta-1)}(x)}{p_n^{(\alpha+r+1,\beta-1)}(1)}\right|^2(1-x)^{s+\alpha}(1+x)^\beta dx\\
& \ls 2^{\alpha+\beta+s}(\alpha+r+1)^{1/3} \Gamma(\alpha+r+2)^{(4\alpha+4s+4)/(2\alpha+2r+3)}n^{-2\alpha-2s-2}.
\end{aligned}
\ee
{\rm (c)} If $\alpha\ge\beta \ge 2$, $r\ge s>-1$, $n\ge (2\alpha+r+1)^2$, $m\ge (\alpha+\beta)^2$, then
\be\label{eq:kernell1norm}
\begin{aligned}
K_m^{(\alpha,\beta)}(1)\int_{-1}^1 &\left|\frac{p_n^{(\alpha+r+1,\beta-1)}(x)}{p_n^{(\alpha+r+1,\beta-1)}(1)}\right|^2(1-x)^{s+\alpha}(1+x)^\beta dx\\
& \ls e^{2r-2s}2^s\left(\frac{m}{n}\right)^{2\alpha+2s+2}m^{-2s} \frac{\Gamma(\alpha+r+2)^2}{\Gamma(\alpha+1)\Gamma(\alpha+2)}(\alpha+r+2)^{2s-2r-2/3}.
\end{aligned}
\ee
\end{theorem}

The proof of this theorem requires some preparation. We will observe a fundamental estimate on the ratio of Gamma functions in Lemma~\ref{lemma:gammaratio}, and apply it in Corollary~\ref{cor:basicest} to obtain some detailed asymptotics for the various quantities introduced in Section~\ref{bhag:jacobiprep}.

\begin{lemma}\label{lemma:gammaratio}
For $x \ge y\ge 1$, we have
\be\label{eq:loggammadiff}
\log\Gamma(x+y)-\log\Gamma(x)=y\log x +\O\left(\frac{y^2}{x}\right).
\ee
In particular, if $x\gs y^2$ then
\be\label{eq:gammaratio}
\frac{\Gamma(x+y)}{\Gamma(x)}\sim x^y.
\ee
\end{lemma}

\begin{proof}\ 

The  Stirling approximation (\cite[Section~4.1]{olverbook}) states that  for $z\ge 1$,
\be\label{eq:pf1eqn1}
\log\Gamma(z)=(z-1/2)\log z -z +(1/2)\log (2\pi) =\O(1/z).
\ee
 Hence, for $x, y\ge 1$,
\be\label{eq:pf1eqn2}
\log\Gamma(x+y)-\log\Gamma(x)=(x+y-1/2)\log(x+y)-(x-1/2)\log x-y + \O\left(\frac{1}{x}\right).
\ee 
Now,
\be\label{eq:pf1eqn3}
(x+y-1/2)\log(x+y)=(x+y-1/2)\log x +x\left(1+\frac{y}{x}\right)\log(1+y/x)-\frac{1}{2}\log(1+y/x).
\ee
 Using Taylor's theorem, a simple calculation shows that
 for $x\ge y\ge 1$,
\be\label{eq:pf1eqn4}
\left(1+\frac{y}{x}\right)\log(1+y/x)=y/x +\O\left( \frac{y^2}{x^2}\right),  \quad \log(1+y/x)=y/x +\O\left( \frac{y^2}{x^2}\right).
\ee
So, we obtain from \eqref{eq:pf1eqn3} that
$$
(x+y-1/2)\log(x+y)=(x+y-1/2)\log x+y+\O\left(\frac{y^2}{x}\right).
$$
This estimate and \eqref{eq:pf1eqn2} lead to \eqref{eq:loggammadiff}.
\end{proof}

\begin{cor}\label{cor:basicest}
For
 $\alpha\ge \beta\ge -1/2$, $n\ge (\alpha+|\beta|)^2$,
\be\label{eq:pnat1asymp}
p_n^{(\alpha,\beta)}(1)\sim \frac{n^{\alpha+1/2}}{2^{(\alpha+\beta)/2}\Gamma(\alpha+1)},
\ee
and
\be\label{eq:jacobikernat1asymp}
K_n^{(\alpha,\beta)}(1)\sim \frac{1}{2^{\alpha+\beta}\Gamma(\alpha+1)\Gamma(\alpha+2)}n^{2\alpha+2},
\ee
where we note that the constants involved in $\sim$ are absolute constants, independent of $n$, $\alpha$, $\beta$.
\end{cor}
\begin{proof}\ 
The estimate \eqref{eq:pnat1asymp} (respectively, \eqref{eq:jacobikernat1asymp}) follows using \eqref{eq:pkat1} (repsectively, \eqref{eq:jacobikernat1}) and Lemma~\ref{lemma:gammaratio}.
\end{proof}\\

The following proposition summarizes some inequalities for Jacobi polynomials. 

\begin{prop}\label{prop:jacobibds} Let $\alpha\ge\beta\ge -1/2$, $n\ge 0$ be an integer.\\
{\rm (a)} We have
\be\label{eq:jacobibd}
\max_{x\in [-1,1]}|p_n^{(\alpha,\beta)}(x)| \le p_n^{(\alpha,\beta)}(1).
\ee
{\rm (b)} If $\alpha\ge\beta\ge (1+\sqrt{2})/4\approx 0.6036$, $n\ge \alpha$, then for $\theta\in (0,\pi)$,
\be\label{eq:jacobiptwisebd}
|p_n^{(\alpha,\beta)}(\cos\theta)|\le 2\alpha^{1/6}(1-\cos\theta)^{-\alpha/2-1/4}(1+\cos\theta)^{-\beta/2-1/4}.
\ee
\end{prop}
\begin{proof}\ 
Part (a) follows from \cite[Theorem~7.32.1]{szego}. Part (b) follows from  \cite[Theorem~2]{krasikov2007upper}.
\end{proof}

\vskip 3pt
We are now in a position to prove Theorem~\ref{theo:jacobisq}.\\

\noindent\textsc{Proof of Theorem~\ref{theo:jacobisq}.}\\
Let $0<\delta<\pi/2$ to be chosen later. Since $b\ge 0$, \eqref{eq:jacobibd} leads to
\be\label{eq:pf3eqn1}
\begin{aligned}
\int_0^{2\delta} |p_n^{(A,B)}(\cos\theta)|^p&(1-\cos\theta)^a(1+\cos\theta)^b\sin\theta d\theta\le 2^{a+b+1}|p_n^{(A,B)}(1)|^p\int_0^{2\delta} \sin(\theta/2)^{2a+1}\cos(\theta/2)^{2b+1}d\theta\\
&\le 2^{a+b+1}|p_n^{(A,B)}(1)|^p\int_0^{2\delta} \sin(\theta/2)^{2a+1}\cos(\theta/2)d\theta \ls \frac{2^{a+b}}{2a+2}|p_n^{(A,B)}(1)|^p(\sin\delta)^{2a+2}.
\end{aligned}
\ee
In view of \eqref{eq:jacobiptwisebd}, and the facts that $2b\ge pB+p/2$, $2a+2<pA+p/2$,
\be\label{eq:pf3eqn2}
\begin{aligned}
\int_{2\delta}^\pi |p_n^{(A,B)}(\cos\theta)|^p& (1-\cos\theta)^a(1+\cos\theta)^b\sin\theta d\theta\ls 2^pA^{p/6}\int_{2\delta}^\pi (1-\cos\theta)^{a-pA/2-p/4}(1+\cos\theta)^{b-pB/2-p/4}\sin\theta d\theta\\
&\ls A^{p/6}2^{a+b-p(A+B-1)/2}\int_{2\delta}^\pi (\sin(\theta/2))^{2a+1-pA-p/2}(\cos(\theta/2))^{2b+1-pB-p/2}d\theta\\
&\ls A^{p/6}2^{a+b-p(A+B-1)/2}\int_\delta^{\pi/2} (\sin\theta)^{2a+1-pA-p/2}\cos\theta d\theta\\
&\ls \frac{A^{p/6}2^{a+b-p(A+B-1)/2}}{pA+p/2-2a-2}(\sin\delta)^{2a+2-pA-p/2}.
\end{aligned}
\ee
We now choose $\delta$ such that
$$
\frac{2^{a+b}}{2a+2}|p_n^{(A,B)}(1)|^p(\sin\delta)^{2a+2}=\frac{A^{p/6}2^{a+b-p(A+B-1)/2}}{pA+p/2-2a-2}(\sin\delta)^{2a+2-pA-p/2};
$$
i.e., 
$$
\sin\delta=\left\{\frac{A^{1/6}}{2^{(A+B-1)/2}p_n^{(A,B)}(1)}\right\}^{1/(A+1/2)}\left\{\frac{2a+2}{p(A+1/2)-2a-2}\right\}^{1/p(A+1/2)}.
$$
(In view of \eqref{eq:ncond} and \eqref{eq:pnat1asymp}, it is not difficult to verify that $|\sin\delta|\le 1$.)
Then 
 \eqref{eq:pf3eqn1} and \eqref{eq:pf3eqn2} lead to
 \eqref{eq:jacobilpintegral}.

Taking $a=\alpha+s$, $b=\beta$, $A=\alpha+r+1$, $B=\beta-1$, and $p=2$, \eqref{eq:jacobilpintegral} becomes
\be\label{eq:pf3eqn3}
\begin{aligned}
\int_{-1}^1 &\left|\frac{p_n^{(\alpha+r+1,\beta-1)}(x)}{p_n^{(\alpha+r+1,\beta-1)}(1)}\right|^2(1-x)^{s+\alpha}(1+x)^\beta dx\\
& \ls  \frac{2^{\alpha+\beta+s}}{2\alpha+2s+2}\left\{\frac{(\alpha+r+1)^{1/6}}{2^{(\alpha+\beta+r-1)/2}p_n^{(\alpha+r+1,\beta-1)}(1)}\right\}^{(2\alpha+2s+2)/(\alpha+r+3/2)}\\
&\qquad\qquad\times \left\{\frac{2\alpha+2s+2}{2r-2s+1}\right\}^{(2\alpha+2s+2)/(2\alpha+2r+3)}
\end{aligned}
\ee
In this proof only, let
\be\label{eq:pf3eqn4}
\gamma=\frac{2\alpha+2s+2}{2\alpha+2r+3}, \quad 1-\gamma=\frac{2r-2s+1}{2\alpha+2r+3}, \quad 0<\gamma<1.
\ee
We have
$$
(2^{1/2}(\alpha+r+1)^{1/6})^{(2\alpha+2s+2)/(\alpha+r+3/2)}= (2^{1/2}(\alpha+r+1)^{1/6})^{2\gamma}\ls (\alpha+r+1)^{1/3}, \qquad (2r-2s+1)^{-\gamma}\le 1,
$$
and
$$
(2\alpha+2s+2)^{(2\alpha+2s+2)/(2\alpha+2r+3)-1}=(2\alpha+2s+2)^{\gamma-1}\ls 1.
$$
Hence, 
using \eqref{eq:pnat1asymp}, \eqref{eq:pf3eqn3} simplifies to \eqref{eq:jacobisqintegral}.

Next, to prove part (c), we use \eqref{eq:jacobikernat1asymp} and \eqref{eq:jacobisqintegral} to deduce that
\be\label{eq:pf3eqn5}
\begin{aligned}
K_m^{(\alpha,\beta)}(1)\left(\int_{-1}^1\right. &\left.\left|\frac{p_n^{(\alpha+r+1,\beta-1)}(x)}{p_n^{(\alpha+r+1,\beta-1)}(1)}\right|^2(1-x)^{s+\alpha}(1+x)^\beta dx\right)^2\\
&\ls \frac{m^{2\alpha+2}}{2^{\alpha+\beta}\Gamma(\alpha+1)\Gamma(\alpha+2)}2^{\alpha+\beta+s}(\alpha+r+1)^{1/3} \Gamma(\alpha+r+2)^{(4\alpha+4s+4)/(2\alpha+2r+3)}n^{-2\alpha-2s-2}\\
&\ls  2^s(\alpha+r+1)^{1/3}\frac{\Gamma(\alpha+r+2)^2}{\Gamma(\alpha+1)\Gamma(\alpha+2)}\Gamma(\alpha+r+2)^{2\gamma-2}  \left(\frac{m}{n}\right)^{2\alpha+2s+2}m^{-2s}.
\end{aligned}
\ee
Using Stirling's approximation, we find that
$$
\begin{aligned}
(1-\gamma)\log\Gamma(\alpha+r+2)&=\frac{2r-2s+1}{2\alpha+2r+3}\log\Gamma(\alpha+r+2)\\
&=(r-s+1/2)\log(\alpha+r+2) -\frac{2r-2s+1}{2\alpha+2r+3}(\alpha+r+2) +\O(1)\\
&\gs (r-s+1/2)\log(\alpha+r+2)-(r-s) -c
\end{aligned}
$$
for some absolute constant $c$.
Therefore, 
$$
\Gamma(\alpha+r+2)^{2\gamma-2} \ls e^{2r-2s}(\alpha+r+2)^{2s-2r-1}.
$$
 The estimate \eqref{eq:kernell1norm} is easy to deduce using this last estimate  in \eqref{eq:pf3eqn5}. 
\qed

\subsection{Estimates on spherical polynomials}\label{bhag:sphereest}

We will use the following proposition without explicitly referring to it.
\begin{prop}\label{prop:sphbasicest}
For $d\ge 1$ and $n\ge d^2$, we have
\be\label{eq:dimasym}
K_{d;n}(1)=\mathsf{dim}(\Pi_n^d)\sim \frac{n^d}{\Gamma(d+1)}.
\ee
We note that the constants involved in $\sim$ are independent of $n$ and $d$.
\end{prop}

\begin{proof}\ 
The relation \eqref{eq:dimasym} follows by using Lemma~\ref{lemma:gammaratio} in \eqref{eq:dimpin}.
\end{proof}

\

Next, we prove the required estimates on the Lebesgue constants for the kernels $\Phi_{d;n,r}$ and $\widetilde{\Phi}_{d;n,r}$.
These estimates will  play a crucial role in the proof of Theorem~\ref{theo:locsphapprox}. 
Therefore, we formulate them as a theorem below.

\begin{theorem}
\label{theo:fundaest}
Let $d\ge 3$, $r\ge s\ge 0$, $nd\ge (d+r+1)^2$, and $\nu\in\mathsf{MZ}(d;2(d+2)n)$. Then
\be\label{eq:lebesgueconst}
\begin{aligned}
\sup_{\x\in\SS^d}\int_{\SS^d}|\widetilde{\Phi}_{d;n,r}(\x\cdot\y)|(1-\x\cdot\y)^{s/2}d|\nu|(\y)&=\sup_{\x\in\SS^d}\int_{\SS^d}|\Phi_{d;n,r}(\x\cdot\y)|(1-\x\cdot\y)^{s/2}\left(\frac{1+\x\cdot\y}{2}\right)^{-n}
d|\nu|(\y)\\
&\ls \frac{d^{1/6}}{n^s}\tn\nu\tn_{d;2(d+2)n}.
\end{aligned}
\ee
\end{theorem}

\begin{proof}\ 
In this proof, we will assume the normalization that $\tn\nu\tn_{d;2(d+2)n}=1$. 
We also assume that $s>0$, and $d\ge 6$.
The case $s=0$ is much simpler; the approximation described in \eqref{eq:pf2eqn1} is then not necessary.
The case $3\le d\le 5$ does not require an elaborate book-keeping as is done here; the same ideas work in a much simpler manner.
Using the direct theorem for approximation by trigonometric polynomials to approximate the function $\theta\mapsto |\sin(\theta/2)|^s$, we obtain an algebraic polynomial $P\in\Pi_n$ such that
\be\label{eq:pf2eqn1}
\max_{x\in [-1,1]}\left|(1-x)^{s/2}-P(x)\right|\ls n^{-s}.
\ee
Since $\nu\in \mathsf{MZ}(d;2(d+2)n)$, we deduce using \eqref{eq:pf2eqn1} and Schwarz inequality that
\be\label{eq:pf2eqn2}
\begin{aligned}
\int_{\SS^d}&|\widetilde{\Phi}_{d;n,r}(\x\cdot\y)|(1-\x\cdot\y)^{s/2}d|\nu|(\y)\ls \int_{\SS^d}|\widetilde{\Phi}_{d;n,r}(\x\cdot\y)||P(\x\cdot\y)|d|\nu|(\y) + n^{-s}\int_{\SS^q}|\widetilde{\Phi}_{d;n,r}(\x\cdot\y)|d|\nu|(\y)\\
&\ls \left\{\int_{\SS^d}K_{d;(d+2)n}(\x\cdot\y)^2d|\nu|(\y)\right\}^{1/2}\left\{\int_{\SS^d} \left|\frac{p_{dn}^{(d/2+r,d/2-2)}(\x\cdot\y)}{p_{dn}^{(d/2+r,d/2-2)}(1)}\right|^2|P(\x\cdot\y)|^2d|\nu|(\y)\right\}^{1/2}\\
&\quad +n^{-s} \left\{\int_{\SS^d}K_{d;(d+2)n}(\x\cdot\y)^2d|\nu|(\y)\right\}^{1/2}\left\{\int_{\SS^d} \left|\frac{p_{dn}^{(d/2+r,d/2-2)}(\x\cdot\y)}{p_{dn}^{(d/2+r,d/2-2)}(1)}\right|^2d|\nu|(\y)\right\}^{1/2}\\
&\ls \left\{\int_{\SS^d}K_{d;(d+2)n}(\x\cdot\y)^2d\mu_d^*(\y)\right\}^{1/2}\left\{\int_{\SS^d} \left|\frac{p_{dn}^{(d/2+r,d/2-2)}(\x\cdot\y)}{p_{dn}^{(d/2+r,d/2-2)}(1)}\right|^2|P(\x\cdot\y)|^2d\mu_d^*(\y)\right\}^{1/2}\\
&\quad +n^{-s}\left\{\int_{\SS^d}K_{d;(d+2)n}(\x\cdot\y)^2d\mu_d^*(\y)\right\}^{1/2}\left\{\int_{\SS^d} \left|\frac{p_{dn}^{(d/2+r,d/2-2)}(\x\cdot\y)}{p_{dn}^{(d/2+r,d/2-2)}(1)}\right|^2d\mu_d^*(\y)\right\}^{1/2}.
\end{aligned}
\ee
We note that \eqref{eq:pf2eqn1} implies that for all $x\in [-1,1]$,
$$
|P(x)|^2\le |P(x)-(1-x)^{s/2}+(1-x)^{s/2}|^2 \le 2|P(x)-(1-x)^{s/2}|^2+2(1-x)^s\ls n^{-2s} +(1-x)^s.
$$
In view of \eqref{eq:sphreprodl2} and\eqref{eq:zonalintegral},  we now conclude that
\be\label{eq:pf2eqn3}
\begin{aligned}
\left\{\int_{\SS^d}\right.&\left.|\widetilde{\Phi}_{d;n,r}(\x\cdot\y)|(1-\x\cdot\y)^{s/2}d|\nu|(\y)\right\}^2\\
& \ls \frac{\omega_{d-1}}{\omega_d}K_{d;(d+2)n}(1)\int_{-1}^1 \left|\frac{p_{dn}^{(d/2+r,d/2-2)}(x)}{p_{dn}^{(d/2+r,d/2-2)}(1)}\right|^2|P(x)|^2(1-x^2)^{d/2-1}dx \\
&\quad+ \frac{\omega_{d-1}}{\omega_d}K_{d;(d+2)n}(1)n^{-2s} \int_{-1}^1 \left|\frac{p_{dn}^{(d/2+r,d/2-2)}(x)}{p_{dn}^{(d/2+r,d/2-2)}(1)}\right|^2(1-x^2)^{d/2-1}dx\\
&\ls\frac{\omega_{d-1}}{\omega_d}K_{d;(d+2)n}(1)\int_{-1}^1 \left|\frac{p_{dn}^{(d/2+r,d/2-2)}(x)}{p_{dn}^{(d/2+r,d/2-2)}(1)}\right|^2(1-x)^{s+d/2-1}(1+x)^{d/2-1}dx \\
&\quad+ \frac{\omega_{d-1}}{\omega_d}K_{d;(d+2)n}(1)n^{-2s} \int_{-1}^1 \left|\frac{p_{dn}^{(d/2+r,d/2-2)}(x)}{p_{dn}^{(d/2+r,d/2-2}(1)}\right|^2(1-x^2)^{d/2-1}dx.
\end{aligned}
\ee
We now recall (cf. \eqref{eq:sphreprodkern}) that
$$
\frac{\omega_{d-1}}{\omega_d}K_{d;(d+2)n}(x)=K_{(d+2)n}^{(d/2-1,d/2-1)}(x),
$$
and
apply Theorem~\ref{theo:jacobisq}(c) with $\alpha=d/2-1$.

If $d<r$ then \eqref{eq:pf2eqn3} and \eqref{eq:kernell1norm} (used once with $s$ and once with $0$ in place of $s$) together yield \eqref{eq:lebesgueconst} directly. 
In the remainder of the proof, we therefore assume that $d\ge r$.

Since $d\ge r$, Lemma~\ref{lemma:gammaratio} implies that
\be\label{eq:pf2eqn4}
\frac{\Gamma(\alpha+r+2)^2}{\Gamma(\alpha+1)\Gamma(\alpha+2)}=\frac{\Gamma(d/2+r+1)^2}{\Gamma(d/2)\Gamma(d/2+1)}\ls d^{2r+1}.
\ee
Using the fact that $d/2+r+1\sim d$ and \eqref{eq:pf2eqn4} in the right hand side of \eqref{eq:kernell1norm} in Theorem~\ref{theo:jacobisq}(c), we obtain that
\be\label{eq:pf2eqn6}
\frac{\omega_{d-1}}{\omega_d}K_{d;(d+2)n}(1)\int_{-1}^1 \left|\frac{p_{dn}^{(d/2+r,d/2-2)}(x)}{p_{dn}^{(d/2+r,d/2-2)}(1)}\right|^2(1-x)^{s+d/2-1}(1+x)^{d/2-1}dx\ls d^{2s+1/3}\left((d+2)n\right)^{-2s}\ls d^{1/3}n^{-2s}.
\ee
We use the same estimate with $s=0$ to obtain
\be\label{eq:pf2eqn7}
\frac{\omega_{d-1}}{\omega_d}K_{d;(d+2)n}(1) \int_{-1}^1 \left|\frac{p_{dn}^{(d/2+r,d/2-2)}(x)}{p_{dn}^{(d/2+r,d/2-2)}(1)}\right|^2(1-x^2)^{d/2-1}dx\ls d^{1/3}.
\ee
The estimates \eqref{eq:pf2eqn3}, \eqref{eq:pf2eqn6}, and \eqref{eq:pf2eqn7} lead to \eqref{eq:lebesgueconst}.
\end{proof}

We end this section with some results on spherical polynomials which will be used in the proofs of various results in Section~\ref{bhag:mainresults}.

The proof of Theorem~\ref{theo:analapprox} requires an an analogue of the so-called Bernstein-Walsh estimate for spherical polynomials; i.e., an estimate on the supremum norm of a spherical polynomial on $\SS^d$ in terms of that on a spherical cap. 
The following lemma was proved in \cite[Lemma~8]{mhaskar2004local2}.
\begin{lemma}\label{lemma:walshlemma}  Let $r\ge 1$ be an integer, $R$ be a trigonometric polynomial of order $\le r$, $0<\gamma<\gamma_1\le \pi/2$, and $\max_{t\in [-2\gamma, 2\gamma]}|R(t)|=1$. Then
\be\label{eq:walshineq}
\max_{t\in [-2\gamma_1,2\gamma_1]}|R(t)|\le \left(\frac{2\sin\gamma_1}{\sin\gamma}\right)^{2r}\le \left(\frac{\pi\gamma_1}{\gamma}\right)^{2r}.
\ee
\end{lemma}

\begin{lemma}\label{lemma:sphwalsh}
Let $r\ge 1$ be an integer, $P\in \Pi_r^d$, $\x\in\SS^q$, $\delta\in (0,2]$. Then
\be\label{eq:sphwalsh}
\|P\|_\infty \le \left(\frac{4}{\delta}\right)^{2r}\|P\|_{\infty,\BB(\x,\delta)}.
\ee
\end{lemma}

\begin{proof}\ 
Since the restriction of $P$ to any geodesic through $\x$ is a trigonometric polynomial of order $r$, the lemma follows directly from Lemma~\ref{lemma:walshlemma}.
\end{proof}

The following lemmas will be used in the proof of Theorem~\ref{theo:probtheo}.

\begin{lemma}\label{lemma:ballvolume}
There exists an absolute constant $C>0$ such that if $0<\rho\le Cd^{-1/2}$, then 
\be\label{eq:ballvolume}
\mu_d^*(\BB(\x,\rho))\sim d^{-1/2}\rho^d, \qquad x\in\SS^d,
\ee
Consequently, $\SS^d$ can be expressed as a union of  $\O(d^{1/2}\rho^{-d})$ balls, each of radius $\rho$.
\end{lemma}

\begin{proof}\ 
We note that if $\theta_0\le \pi/2$ is such that $\rho=2\sin(\theta_0/2)$, then 
\be\label{eq:pf5eqn1}
\mu_d^*(\BB(\x,\rho))=\frac{\omega_{d-1}}{\omega_d}\int_0^{\theta_0}\sin^{d-1}\theta d\theta.
\ee
In view of \eqref{eq:sphvol} and Stirling's approximation, we have
\be\label{eq:pf5eqn2}
\frac{\omega_{d-1}}{\omega_d}\sim d^{1/2}.
\ee
We note further the elementary facts that for any $\alpha\ge 1$,
\be\label{eq:pf5eqn3}
0\le 1-(1-x)^\alpha\ls \alpha x, \qquad 0\le x\le 1/4,
\ee
so that
\be\label{eq:pf5eqn4}
|\sin^\alpha\theta_0-\rho^\alpha| =\rho^\alpha\left|1-(1-\rho^2/4)^\alpha\right| \ls \alpha\rho^{\alpha+2},
\ee
where the constants involved in both the above inequalities are independent of $\alpha$.
Finally, since the function $x\mapsto (\sin x)/x$ is decreasing on $x\in[0,\pi/2]$, we deduce that
\be\label{eq:pf5eqn5}
 \frac{\sin^{\alpha+1}\theta_0}{\alpha+1}= \frac{\sin^{\alpha}\theta_0}{\theta_0^\alpha}\frac{\theta_0^{\alpha+1}}{\alpha+1}\le \int_0^{\theta_0} \sin^\alpha\theta d\theta \le (\sin^\alpha\theta_0)\theta_0\le \frac{2}{\pi}\sin^{\alpha+1}\theta_0.
\ee
Using integration by parts, we see that
$$
\int_0^{\theta_0}\sin^{d-1}\theta d\theta=\frac{\sin^d\theta_0}{d}\cos\theta_0+\frac{d+1}{d}\int_0^{\theta_0}\sin^{d+1}\theta d\theta=\frac{\rho^d}{d}\left(1-\frac{\rho^2}{4}\right)^{d/2}\left(1-\frac{\rho^2}{2}\right)+\frac{d+1}{d}\int_0^{\theta_0}\sin^{d+1}\theta d\theta.
$$
Hence,
$$
\frac{d}{\rho^d}\int_0^{\theta_0}\sin^{d-1}\theta d\theta= \left(1-\frac{\rho^2}{4}\right)^{d/2}\left(1-\frac{\rho^2}{2}\right)+\frac{d+1}{\rho^d}\int_0^{\theta_0}\sin^{d+1}\theta d\theta.
$$
The estimates \eqref{eq:pf5eqn3},  \eqref{eq:pf5eqn4}, and \eqref{eq:pf5eqn5} now lead to
$$
\left|1-\frac{d}{\rho^d}\int_0^{\theta_0}\sin^{d-1}\theta d\theta\right|\ls d\rho^2.
$$
This proves \eqref{eq:ballvolume}. 
\end{proof}

\begin{lemma}\label{lemma:polycover}
With $C$ as in Lemma~\ref{lemma:ballvolume}, let $n\ge d^{1/2}\pi/(4C)$. Then there exists a finite set $\mathcal{C}\subset \SS^d$ with $|\mathcal{C}|\sim d^{1/2}(4n/\pi)^d$ with the property that for every $P\in\Pi_n^d$,
\be\label{eq:polymax}
 (1/2)\|P\|_\infty\le \max_{\z\in\mathcal{C}}|P(\z)|\le \|P\|_\infty.
\ee
\end{lemma}

\begin{proof}\ 
Let $\mathcal{C}$ be a minimal set such that $\SS^d=\bigcup_{\z\in\mathcal{C}}\BB(\z, \pi/(4n))$. In view of Lemma~\ref{lemma:ballvolume}, $|\C|\sim d^{1/2}(4n/\pi)^d$. 
Let $P\in \Pi_n^d$, and $\|P\|_\infty=|P(\x^*)|$ for some $\x^*\in\SS^d$. 
Since the restriction of $P$ to any geodesic is a trigonometric polynomial of degree $<n$, the Bernstein inequality for trigonometric polynomials yields
$$
\bigg||P(\y)|-|P(\x^*)|\bigg|\le |P(\y)-P(\x^*)|\le n\cos^{-1}(\x^*\cdot\y)\|P\|_\infty\le \frac{2n}{\pi}|\x^*-\y|_{d+1}\|P\|_{\infty}, \qquad \y\in\SS^d.
$$
Since there exists $\z\in\mathcal{C}$ such that $|\x^*-\z|_{d+1}\le \pi/(4n)$, the estimate \eqref{eq:polymax} is now clear.
\end{proof}
\subsection{Proofs of the theorems in Sections~\ref{bhag:kernels} and \ref{bhag:mainresults}.}\label{bhag:mainproofs}
We start with the proof of Theorem~\ref{theo:goodsphapprox}. \\[1ex]

\noindent\textsc{Proof of Theorem~\ref{theo:goodsphapprox}.}
To prove part (a), let $\x\in\SS^d$. The polynomial $R$ defined by
$$
R(\y)=\frac{p_{dn}^{(d/2+r,d/2-2)}(\x\cdot\y)}{p_{dn}^{(d/2+r,d/2-2)}(1)}\left(\frac{1+\x\cdot\y}{2}\right)^nP(\y), \qquad \y\in\SS^d
$$
is in $\Pi_{(d+2)n}^d$ and satisfies $R(\x)=P(\x)$. 
In view of \eqref{eq:sphreprod} applied with $R$ in place of $P$, and the fact that $\nu$ is a quadrature measure of order $2(d+2)n$, we get
$$
\sigma_{d;n,r}(\nu, P)(\x)=\int_{\SS^d}K_{d;(d+2)n}(\x\cdot\y)R(\y)d\nu(\y)=\int_{\SS^d}K_{d;(d+2)n}(\x\cdot\y)R(\y)d\mu^*_d(\y)=R(\x)=P(\x).
$$
The proof that $\widetilde{\sigma}_{d;n,r}(P)=P$ is similar. This proves part (a).

Next, we observe that 
the first estimate in \eqref{eq:degapproxstrong} is obvious since $\widetilde{\sigma}_{d;n,r}(\nu,f)\in \Pi_{2(d+2)n}^d$.
 Using Theorem~\ref{theo:fundaest} with $s=0$, it is easy to deduce that
$$
\|\widetilde{\sigma}_{d;n,r}(\nu,g)\|_\infty\le \|\sigma_{d;n,r}(\nu,g)\|_\infty\ls d^{1/6}\tn\nu\tn_{d;2(d+2)n}\|g\|_\infty, \qquad g\in C(\SS^q).
$$
In view of Theorem~\ref{theo:goodsphapprox}(a), for any $P\in\Pi_n^d$, we have
$$
\|f-\widetilde{\sigma}_{d;n,r}(\nu,f)\|_\infty=\|f-P-\widetilde{\sigma}_{d;n,r}(\nu,f-P)\|_\infty\le \|f-P\|_\infty +\|\widetilde{\sigma}_{d;n,r}(\nu,f-P)\|_\infty\ls 
d^{1/6}\tn\nu\tn_{d;2(d+2)n}\|f-P\|_\infty.
$$
The second estimate in \eqref{eq:degapproxstrong} is now clear.
The estimate \eqref{eq:degapproxstrongbis} is proved similarly.
\qed

Next, we prove Theorem~\ref{theo:locsphapprox}.

\vskip6pt
\noindent\textsc{Proof of Theorem~\ref{theo:locsphapprox}.}
In this proof, we  assume without loss of generality that $\|f\|_{W_{d;r,\x}}=1$,  and write $\delta$ in place of $\delta_n$.
 Let $P\in \Pi_r^d$ satisfy
\be\label{eq:pf4eqn2}
|f(\y)-P(\y)| \le |\x-\y|_{d+1}^r =2^r(1-\x\cdot\y)^{r/2}, \qquad \y\in \BB(\x,\delta).
\ee
Clearly,   $\|P\|_{\infty, \BB(\x,\delta)}\ls \|f\|_\infty \le 1$.
In view of Lemma~\ref{lemma:sphwalsh}, we see that
\be\label{eq:pf4eqn3}
\|f-P\|_\infty\le \|f\|_\infty +\|P\|_\infty\le \|f\|_\infty + (4/\delta)^{2r}\|P\|_{\infty,\BB(\x,\delta)} \ls (4/\delta)^{2r}.
\ee
Moreover, for $\y\in \SS^d\setminus \BB(\x,\delta)$,
\be\label{eq:pf4eqn4}
\frac{1+\x\cdot\y}{2}=1-\frac{|\x-\y|_{d+1}^2}{4}\le \exp\left(-\frac{|\x-\y|_{d+1}^2}{4}\right)\le \exp(-\delta^2/4).
\ee
Using \eqref{eq:pf4eqn3}, \eqref{eq:pf4eqn4},  Theorem~\ref{theo:fundaest} (with $0$ in place of $s$), and the definition 
$$
\delta=\sqrt{\frac{16r\log n}{n}}
$$
we deduce that
\be\label{eq:pf4eqn1}
\begin{aligned}
 \int_{\SS^d\setminus \BB(\x,\delta)}|\Phi_{d;n,r}(\x\cdot\y)|&|f(\y)-P(\y)|d|\nu|(\y)\\
 &\le \exp(-n\delta^2/4)\|f-P\|_\infty \int_{\SS^d}|\Phi_{d;n,r}(\x\cdot\y)|\left(\frac{1+\x\cdot\y}{2}\right)^{-n}d|\nu|(\y)\\
&\ls d^{1/6}\tn\nu\tn_{d;2(d+2)n}\delta^{-2r}\exp\left(-n\delta^2/4 \right)\\
&\ls d^{1/6}(16r\log n)^{-r}n^{-3r}\tn\nu\tn_{d;2(d+2)n}.
\end{aligned}
\ee
Similarly,
\be\label{eq:pf4eqn6}
\int_{\SS^d\setminus \BB(\x,\delta)}|\Phi_{d;n,r}(\x\cdot\y)||f(\y)|d|\nu|(\y)\ls d^{1/6} n^{-4r}\tn\nu\tn_{d;2(d+2)n}.
\ee
Next, we note that since $P\in\Pi_r^d\subset \Pi_n^d$,
\eqref{eq:pf4eqn2} implies that 
$$
f(\x)=P(\x)=\sigma_{d;n,r}(\nu,P)(\x).
$$ 
Consequently, using Theorem~\ref{theo:goodsphapprox}(a), \eqref{eq:pf4eqn2}, \eqref{eq:pf4eqn1}, and Theorem~\ref{theo:fundaest} (with $r$ in place of $s$), we obtain
\be\label{eq:pf4eqn5}
\begin{aligned}
|f(\x)-\sigma_{d;n,r}(\nu,f)(\x)|&=|\sigma_{d;n,r}(\nu, f-P)(\x)|\le  \int_{\SS^d}|\Phi_{d;n,r}(\x\cdot\y)||f(\y)-P(\y)|d|\nu|(\y)\\
&\le  \int_{\BB(\x,\delta)}|\Phi_{d;n,r}(\x\cdot\y)||f(\y)-P(\y)|d|\nu|(\y)\\
&\qquad\qquad + \int_{\SS^d\setminus \BB(\x,\delta)}|\Phi_{d;n,r}(\x\cdot\y)||f(\y)-P(\y)|d|\nu|(\y)\\
&\ls \int_{\SS^d}|\Phi_{d;n,r}(\x\cdot\y)|(1-\x\cdot\y)^{r/2}d|\nu|(\y)+d^{1/6}\tn\nu\tn_{d;2(d+2)n}(16r\log n)^{-r}n^{-3r}\\
&\ls \frac{d^{1/6}}{n^r}\tn\nu\tn_{d;2(d+2)n}+d^{1/6}(16r\log n)^{-r}n^{-3r}\tn\nu\tn_{d;2(d+2)n}.
\end{aligned}
\ee
Since 
(cf. Proposition~\ref{prop:sphbasicest}, and Stirling approximation),
$$
(\mathsf{dim}(\Pi_{2(d+2)n}^d))^{1/d}\sim n,
$$
the estimates \eqref{eq:loc_deg_wholeint} follows from \eqref{eq:pf4eqn5}. 
Further, since
$$
\int_{\BB(x,\delta_n)}\Phi_{d;n,r}(\x\cdot\y)f(\y)d\nu(\y)=\sigma_{d;n,r}(\nu,f)(\x)-\int_{\SS^d\setminus \BB(\x,\delta)}\Phi_{d;n,r}(\x\cdot\y)f(\y)d\nu(\y),
$$
and 
the estimate \eqref{eq:loc_deg_approx} follows from \eqref{eq:loc_deg_wholeint} and \eqref{eq:pf4eqn6}. \qed\\[1ex]

 Theorem~\ref{theo:globalsphapprox} is now very easy to  prove.

\vskip 6pt
\noindent\textsc{Proof of Theorem~\ref{theo:globalsphapprox}.}
We note that the definition of $\delta(d;f)$ and the condition on $n$ implies that the condition \eqref{eq:near_ball_def} holds with $\delta(d;f,\x)=\delta(d;f)$ for all $\x\in \SS^d$. Therefore, the theorem follows from Theorem~\ref{theo:locsphapprox}. \qed\\[1ex]

Next, we prove Theorem~\ref{theo:analapprox}.\\[1ex]

\noindent\textsc{Proof of Theorem~\ref{theo:analapprox}.}
The proof is very similar to that of Theorem~\ref{theo:locsphapprox}. 
We sketch the changes.
We  assume without loss of generality that $\|f\|_{A(d;r,\x)}=1$,  and write $\delta=\delta(d;f,\x)$.
Then we choose $P\in\Pi_n$ such that
\be\label{eq:pf7eqn1}
\|f-P\|_{\infty,\BB(\x,\delta)}\le \exp(-n\rho).
\ee
Then, as before,
\be\label{eq:pf7eqn2}
\|f-P\|_\infty\le \|f\|_\infty+(4/\delta)^{2n}\|P\|_{\infty,\BB(\x,\delta)}\ls (4/\delta)^{2n}.
\ee 
Using \eqref{eq:pf4eqn4}, we conclude as in \eqref{eq:pf4eqn1} that
\be\label{eq:pf7eqn3}
\int_{\SS^d\setminus \BB(\x,\delta)}|\Phi_{d;n,r}(\x\cdot\y)||f(\y)-P(\y)|d|\nu|(\y)\ls d^{1/6}\tn\nu\tn_{d;2(d+2)n}\left\{ (4/\delta)^{2n}\exp(-n\delta^2/4)+\exp(-n\rho)\right\}.
\ee
The rest of the proof is almost verbatim the same as that of Theorem~\ref{theo:locsphapprox} with obvious changes.  \qed

The main idea behind the proof of Theorem~\ref{theo:probtheo} is to use the following concentration inequality.  
This inequality is stated below as
Proposition~\ref{prop:concentration}, and  is a reformulation of  \cite[Section~2.1, 2.7]{boucheron2013concentration}.

\begin{prop}\label{prop:concentration}
 (\textbf{Bernstein concentration inequality}) Let $Z_1,\cdots, Z_M$ be independent real valued random variables such that for each $j=1,\cdots,M$, $|Z_j|\le R$, and $\mathbb{E}(Z_j^2)\le V$. Then for any $t>0$,
\be\label{bernstein_concentration}
\mathsf{Prob}\left( \left|\frac{1}{M}\sum_{j=1}^M (Z_j-\mathbb{E}(Z_j))\right| \ge Vt/R\right) \le 2\exp\left(-\frac{MVt^2}{2R^2(1+t)}\right).
\ee
\end{prop}

A straightforward application with $Z_j=f(\y_j)\Phi_{d;n,r}(\z\cdot\y_j)$, $\z\in \SS^d$, would give the points $\y_j$ dependent on $\z$. 
We will  use a covering argument (Lemma~\ref{lemma:polycover}) to obtain bounds on the supremum norm of $\hat{\sigma}_{d;n,r}(\C,f)$ defined  in \eqref{eq:probestimator}. \\

\noindent\textsc{Proof of Theorem~\ref{theo:probtheo}.}
We will first fix $\z\in\SS^d$, and apply Proposition~\ref{prop:concentration} with the random variables
$$
Z_j=f(\y_j)\Phi_{d;n,r}(\z\cdot\y_j).
$$
In view of the definition \eqref{eq:jacksonkerndef}, it is clear that 
\be\label{eq:pf6eqn1}
|Z_j|\le \|f\|_\infty K_{d;(d+2)n}(1), \qquad j=1,\cdots,M.
\ee
Since $d\tau=f_0d\mu_d^*$, it is clear that
\be\label{eq:pf6eqn2}
\mathbb{E}_\tau(Z_j)=\int_{\SS^d} f(\y)\Phi_{d;n,r}(\z\cdot\y)d\tau(\y)=\sigma_{d;n,r}(\mu_d^*,f_0f)(\z).
\ee
Further, using \eqref{eq:sphreprodl2}, we see that
\be\label{eq:pf6eqn3}
\begin{aligned}
\mathbb{E}_\tau(Z_j^2)&=\int_{\SS^d} f(\y)^2\Phi_{d;n,r}(\z\cdot\y)^2d\tau(\y)=\int_{\SS^d} f_0(\y)f(\y)^2\Phi_{d;n,r}(\z\cdot\y)^2d\mu_d^*(\y)\\
&\le \|f_0f\|_\infty\|f\|_\infty\int_{\SS^d}K_{d;(d+2)n}(\z\cdot\y)^2d\mu_d^*(\y)=\|f_0f\|_\infty\|f\|_\infty K_{d;(d+2)n}(1).
\end{aligned}
\ee
Hence, Proposition~\ref{prop:concentration} implies that for each $\z\in\SS^d$ and $t>0$, 
\be\label{eq:pf6eqn4}
\mathsf{Prob}\bigg(\left|\widehat{\sigma}_{d;n,r}(\C,f)(\z)-\sigma_{d;n,r}(\mu_d^*,f_0f)(\z)\right|\ge \|f_0f\|_\infty t\bigg)\le 2\exp\left(\frac{-M\|f_0f\|_\infty t^2}{2\|f\|_\infty K_{d;(d+2)n}(1)(1+t)}\right).
\ee
Since $\widehat{\sigma}_{d;n,r}(\C,f)-\sigma_{d;n,r}(\mu_d^*,f_0f)\in \Pi_{2(d+2)n}^d$, we may use Lemma~\ref{lemma:polycover} to obtain a set $\{\z_1,\cdots,\z_N\}$ with $N\sim d^{1/2}(8(d+2)n/\pi)^d$ such that
\be\label{eq:pf6eqn5}
\begin{aligned}
 (1/2)\|\widehat{\sigma}_{d;n,r}(\C,f)-\sigma_{d;n,r}(\mu_d^*,f_0f)\|_\infty &\le \max_{1\le k\le N}\left|\widehat{\sigma}_{d;n,r}(\C,f)(\z_j)-\sigma_{d;n,r}(\mu_d^*,f_0f)(\z_j)\right|\\
 &\le \|\widehat{\sigma}_{d;n,r}(\C,f)-\sigma_{d;n,r}(\mu_d^*,f_0f)\|_\infty.
\end{aligned}
\ee
In view of Proposition~\ref{prop:sphbasicest} and Stirling's approximation, we have
\be\label{eq:pf6eqn6}
N\sim \left(\frac{16ed}{\pi}\right)^{d+1}K_{d;2(d+2)n}(1)=\left(\frac{16ed}{\pi}\right)^{d+1}D_n, \qquad K_{d;(d+2)n}(1)\sim 2^{-d}K_{d;2(d+2)n}=2^{-d}D_n.
\ee
Using \eqref{eq:pf6eqn4} with each $z_j$, we conclude that for each $t>0$,
\be\label{eq:pf6eqn7}
\mathsf{Prob}\bigg(\|\widehat{\sigma}_{d;n,r}(\C,f)-\sigma_{d;n,r}(\mu_d^*,f_0f)\|_\infty\ge 2\|f_0f\|_\infty t\bigg)\ls \left(\frac{16ed}{\pi}\right)^{d+1}K_{d;2(d+2)n}\exp\left(\frac{-M2^d\|f_0f\|_\infty t^2}{2\|f\|_\infty  D_n(1+t)}\right).
\ee
Setting $t=D_n^{-r/d}$ we see that the right hand side of \eqref{eq:pf6eqn7} is $\le \epsilon$ if \eqref{eq:Meqn} is satisfied with a suitable constant. 
Thus, with probability $\ge 1-\epsilon$, we have
$$
\|\widehat{\sigma}_{d;n,r}(\C,f)-\sigma_{d;n,r}(\mu_d^*,f_0f)\|_\infty\ls \|f_0f\|_\infty D_n^{-r/d}\le \|f_0f\|_{W_{d;r}}D_n^{-r/d}.
$$
Together with Theorem~\ref{theo:globalsphapprox}(b) used with $f_0f$ in place of $f$, we now deduce that the estimate \eqref{eq:probest} holds with probability $\ge 1-\epsilon$.
The equivalent formulation can be derived by a little tedious but simple computation using Proposition~\ref{prop:sphbasicest} and the solution of an equation involving Lambert functions \cite[Lemma~6.1]{witnesspaper}. \qed

\begin{appendix}
\renewcommand{\theequation}{\Alph{section}.\hindu{equation}}
\bhag{Comments on computation of the kernel}\label{bhag:computation}

We make some remarks about the computation of the kernel $\Phi_{d;n,r}$. 
First, we recall that the orthonormalized Jacobi polynomials satisfy the recurrence relations
\be\label{eq:jacobirec1}
\frac{1+t}{2}p_k^{(\alpha,\beta)}(t)=\rho_k^{(\alpha,\beta)} p_{k+1}^{(\alpha,\beta)}(t)+d_k^{(\alpha,\beta)} 
p_k^{(\alpha,\beta)}(t)+\rho_{k-1}^{(\alpha,\beta)}p_{k-1}^{(\alpha,\beta)}(t),
\ee
with 
\be\label{eq:jacobiatzero}
p_{-1}^{(\alpha,\beta)}(t)=0, \quad p_0^{(\alpha,\beta)}(t)=\sqrt{\frac{\Gamma(\alpha+\beta+2)}{2^{\alpha+\beta+1}\Gamma(\alpha+1)\Gamma(\beta+1)}}, 
\ee
where
\be\label{eq:jacobireccoeff1} 
\rho_0^{(\alpha,\beta)}:=\frac{1}{\alpha+\beta+2}\sqrt{\frac{(\alpha+1)(\beta+1)}{\alpha+\beta+3}}, \ d_0^{(\alpha,\beta)}:=\frac{1}{2}+\frac{\beta-\alpha}{2\alpha+2\beta+3},
\ee
and for $k=1,2,\cdots$,
\bea\label{eq:jacobireccoeff2}
\rho_k^{(\alpha,\beta)}&:=&\sqrt{\frac{(k+1)(k+\alpha+1)(k+\beta+1)(k+\alpha+\beta+1)}{(2k+\alpha+\beta+1)(2k+\alpha+\beta+2)^2(2k+\alpha+\beta+3)}},\nonumber\\
d_k^{(\alpha,\beta)}&:=&\frac{1}{2}+\frac{\beta^2-\alpha^2}{2(2k+\alpha+\beta)(2k+\alpha+\beta+1)}.
\eea
The quantity 
$\disp
K_{d;(d+2)n}(x)\left(\frac{1+x}{2}\right)^n
$
can be computed using \eqref{eq:sphreprodkern} and the recurrence relations repeatedly.

Writing
$$
R_n^{(\alpha,\beta)}(x)=\frac{p_n^{(\alpha,\beta)}(x)}{p_n^{(\alpha,\beta)}(1)},
$$
the recurrence relations for these are:
\be\label{eq:gasper_rec}
\begin{aligned}
(2n+\alpha+\beta-1)&(2n+\alpha+\beta)(2n+\alpha+\beta-2)xR_{n-1}(x)\\
&=2(n+\alpha+\beta)(2n+\alpha+\beta-2)(n+\alpha)R_n(x)-(2n+\alpha+\beta-1)(\alpha^2-\beta^2)R_{n-1}(x)\\
&\qquad+2(2n+\alpha+\beta)(n+\beta-1)(n-1)R_{n-2}(x), \qquad n=1,2,\cdots,
\end{aligned}
\ee
with the initial conditions
\be\label{eq:gasper_init}
R_{-1}(x)=0, \qquad
R_0(x)=1.
\ee

Finally, let $\{P_k\}$, $\{\tilde{P}_k\}$ be families of orthogonal polynomials (not necessarily normalized) satisfying
$$
\begin{aligned}
xP_k(x)&=\rho_k P_{k+1}(x)+d_kP_k(x)+r_kP_{k-1}(x),\\
x\tilde{P}_k(x)&=\tilde\rho_k P_{k+1}(x)+\tilde d_kP_k(x)+\tilde{r}_kP_{k-1}(x),
\end{aligned}
$$
with 
$$
P_{-1}(x)=\tilde P_{-1}(x)\equiv 0, \qquad P_0, \tilde P_0 \mbox{ constants}.
$$
Let 
$$
P_k\tilde{P}_j=\sum_{\ell=0}^\infty C(\ell;k,j)P_\ell,
$$
where $C(\ell;k,j)=0$ if $\ell<0$ or $\ell>k+j$ or $k<0$ or $j<0$. 
Then we have the Gautschi 5-point recurrence
$$
C(\ell;k,j+1)=\frac{1}{\tilde{\rho}_j}\left\{\rho_{\ell-1}C(\ell-1;k,j)+(d_\ell-\tilde{d}_j)C(\ell;k,j) +r_{\ell+1} C(\ell+1;k,j)-\tilde{r}_j C(\ell;k,j-1)\right\}.
$$
Together with \eqref{eq:jacobirec1} and \eqref{eq:gasper_rec}, this helps to compute $\Phi_{d;n,r}$ in terms of $\{p_\ell^{(d/2,d/2-1)}\}_{\ell=0}^{2(d+2)n}$.

\bhag{Example}\label{bhag:non_chen}
The purpose of this example is to show that the branch and trunk network approach in the paper \cite{chen1995universal} might not always be the best way to achieve a good degree of approximation.

Let $\mathfrak{X}$ be a separable Hilbert space with inner product $\langle \circ,\circ\rangle$ (with the corresponding norm $\|\circ\|_{\mathfrak{X}}$), and $\{p_j\}_{j=0}^\infty$ be an orthonormal basis for the space. 
For $F\in\mathfrak{X}$, let $\hat{F}(j)=\langle F, p_j\rangle$.
Let $s>1/2$, and 
$$
K_\mathfrak{X}=\{F\in \mathfrak{X}: \sum_{j=1}^\infty j^{2s}|\hat{F}(j)|^2\le 1\}.
$$
It is not difficult to prove using Lemma~\ref{lemma:ballvolume} below that for any $t\in (0,1)$,  $K_\mathfrak{X}$ is contained in a union of\\
 $\disp\O\left(t^{-1/(2s)}\exp((2/t)^{1/s}\log(2/t))\right)$ balls of radius at most $t$.
For $d\ge 1$, we take $\mathcal{I}_{d,K_\mathfrak{X}}(F)=(\hat{F}(j))_{j=0}^{d-1}\in\RR^d$, and define $\mathcal{A}_{d,K_\mathfrak{X}}(\mathbf{a})=\sum_{j=0}^{d-1}a_jp_j$, $\mathbf{a}=(a_0,\cdots,a_{d-1})\in\RR^d$.
It is easy to see that
$$
\mathsf{wor}(K_\mathfrak{X};\mathcal{A}_{d,K_\mathfrak{X}}, \mathcal{I}_{d,K_\mathfrak{X}})\le d^{-s}.
$$
The range of $\mathcal{I}_{d,K_\mathfrak{X}}$ is a subset of the unit ball of $\RR^d$, so that we may take $K_S$ to be this ball.

The operator in this example is motivated by neural operators \cite{li2020neural, kovachki2021universal}.
In analogy to neural networks where each layer acts by taking multiplying the input vector by a matrix, these are defined by applying a linear operator to the input function.
Thus, each layer with activation function $\sigma$ evaluates
$$
x\mapsto \sigma\left((K_\theta F)(x)+b(x)\right),
$$
where $K_\theta$ is one member of a parametrized family of operators, $b$ is a  function that plays the role of the threshold in usual neural networks, and $F$ is the input function. 
In neural Fourier operators, the operator $K_\theta$ is a convolution operator.
Our operator in the example is a generalization and abstraction of this idea.

We choose $\mathfrak{Y}=C([-1,1])$, and define the operator $\mathcal{F}: K_\mathfrak{X}\to C([-1,1])$ as follows.
Let $G_{out}:[-1,1]\times [-1,1]\to [0,1]$, $G_{in}:\mathfrak{X}\times\mathfrak{X}\to [-1,1]$ be  Lipschitz continuous functions, and $\tau$ be a probability measure on $K_\mathfrak{X}$. 
We define
\be\label{eq:antichenoperator}
\mathcal{F}(F)(y)=\int_{K_\mathfrak{X}}G_{out}(y,G_{in}(F, g))d\tau(g).
\ee
Thus, for example, in analogy to neural Fourier operator, $G_{in}(F,g)$ is an inner product of $\hat{F}$ with $\hat{g}$,\\ $G_{out}(y, G_{in}(F,g))=\sigma(b(y)+G_{in}(F,g))$, and the sum which appears implicitly with different parameters in the convolution kernel is replaced by an integral over $g$ with respect to a general  probability measure.
Using  Hoeffding's inequality \cite[Theorem~2.8]{boucheron2013concentration}, one can  deduce using the same ideas as in the  proof of Theorem~\ref{theo:probtheo} in this paper that there exist $g_1,\cdots, g_m\in K_\mathfrak{X}$ such that
$$
 \max_{y\in [-1,1], \ F\in K_\mathfrak{X}}\left|\mathcal{F}(F)(y)-\frac{1}{m}\sum_{j=1}^m G_{out}(y,G_{in}(F, g_j))\right| \ls \left(\frac{\log m}{m}\right)^{s/(2s+1)}.
$$

So, we take $\mathcal{I}_{m,K_\mathfrak{Y}}(\mathcal{F}(F))=(G_{in}(F, g_j))_{j=1}^m$ and the reconstruction algorithm to be
$$
\mathcal{A}_{m,K_\mathfrak{Y}}(\mathbf{a})=\frac{1}{m}\sum_{j=1}^m G_{out}(y,a_j), \qquad \mathbf{a}=(a_1,\cdots, a_m)\in [-1,1]^m.
$$
Then
$$
\mathsf{wor}(K_\mathfrak{Y};\mathcal{A}_{m,K_\mathfrak{Y}}, \mathcal{I}_{m,K_\mathfrak{Y}}) \ls \left(\frac{\log m}{m}\right)^{s/(2s+1)}.
$$

For $j=1,\cdots, m$, we may define $f_j :K_S\to \RR$, by $f_j(\mathbf{a})=G_{in}\left(\mathcal{A}_{d,K_\mathfrak{X}}(\mathbf{a}), g_j\right)$.
The main difficulty in approximating the operator $\mathcal{F}$ is to find an approximation operator $\mathbb{G}_{d,N}: K_S\to \RR$ to approximate each $f_j$.
We note that the information operators $\mathcal{I}_{m,K_\mathfrak{Y}}$ are not continuous.
Moreover, it is more natural to use an approximation of the form $\sum_k w_kG_{out}(\circ,a_k)$ directly rather than taking an eigendecomposition of $G$ as a branch and trunk approach would require.
\end{appendix}



\end{document}